\documentclass{amsart}
\usepackage{latexsym,amsxtra,amscd,ifthen}
\usepackage{amsfonts}
\usepackage{verbatim}
\usepackage{amsmath}
\usepackage{amsthm}
\usepackage{amssymb}
\usepackage{url,color}
\usepackage[arrow,matrix]{xy}

\allowdisplaybreaks[4]

\newtheorem{theorem}{Theorem}[section]

\newtheorem{corollary}[theorem]{Corollary}
\newtheorem{lemma}[theorem]{Lemma}

\newtheorem{claim}[theorem]{Claim}

\newtheorem{proposition}[theorem]{Proposition}
\newtheorem{prop-def}[theorem]{Proposition-Definition}
\theoremstyle{definition}
\newtheorem{definition}[theorem]{Definition}
\newtheorem{remark}[theorem]{Remark}

\newtheorem{example}[theorem]{Example}
\numberwithin{equation}{section}
\def\la{\lambda}

\def\U{\mathcal{U}}
\def\P{\mathcal{P}}
\def\K{\Bbbk}
\def\ot{\otimes}
\def\deg{{\rm deg}}
\def\al{\alpha}

\def\xr{\xrightarrow}

\def\N{\mathbb{N}}

\def\id{{\rm id}}

\def\Sum{\sum\limits}
\def\Ker{{\rm Ker}}
\def\Der{{\rm Der}}
\def\1{\mathbbold{1}}

\def\ad{{\rm ad}}

\def\d{{\rm d}}
\def\fm{\mathfrak m}

\newcommand{\pp}[1]{{#1}^{\{p\}}}
\newcommand{\pk}[1]{{#1}^{\{3\}}}

\newcommand{\pian}[2]{\dfrac{\partial #1}{\partial #2}}

\begin{document}
\title[Restricted Poisson Algebras]
{Restricted Poisson Algebras
}

\author{Y.-H. Bao, Y. Ye and J.J. Zhang}

\address{Bao: School of Mathematical Sciences,
Anhui University, Hefei, 230601, China}

\email{baoyh@ahu.edu.cn, yhbao@ustc.edu.cn}

\address{Ye: School of Mathematical Sciences,
University of Sciences and Technology of China,
Hefei, 230026, China $^1$\\
Wu Wen-Tsun Key Laboratory of Mathematics, USTC,
Chinese Academy of Sciences,
Hefei, 230026, China $^2$}

\email{yeyu@ustc.edu.cn}

\address{Zhang: Department of Mathematics, Box 354350,
University of Washington, Seattle, Washington 98195, USA}

\email{zhang@math.washington.edu}




\begin{abstract}
We re-formulate Bezrukavnikov-Kaledin's definition of a restricted Poisson algebra,
provide some natural and interesting examples, and discuss connections with other
research topics.
\end{abstract}

\maketitle

\dedicatory{}%
\commby{}%

\setcounter{section}{-1}
\section{Introduction}
\label{xxsec0}

The Poisson bracket was introduced by Poisson as a tool for classical
dynamics in 1809 \cite{Po}. Poisson geometry has become an active
research field during the past 50 years. The study of Poisson algebras
over ${\mathbb R}$ or a field of characteristic zero \cite{L-GPV} also
has a long history, and is closely related to noncommutative algebra,
differential geometry, deformation quantization, number theory,
and other areas. The notion of a \emph{restricted Poisson algebra} was
introduced about ten years ago in an important paper of
Bezrukavnikov-Kaledin \cite{BK} in the study of deformation
quantization in positive characteristic. The project in \cite{BK}
is a natural extension of the classical deformation quantization
of symplectic (or Poisson) manifolds.

Our first goal is to better understand Bezrukavnikov-Kaledin's
definition via a Lie algebraic
approach. We re-interpret their definition in the following way.

Throughout the paper let $\K$ be a base field of characteristic
$p\geq 3$. All vector spaces and algebras are over $\K$.

\begin{definition}
\label{xxdef0.1}
Let $(A, \{-,-\})$ be a Poisson algebra over $\Bbbk$.
\begin{enumerate}
\item[(1)]
We call $A$ a {\it weakly restricted Poisson algebra} if
there is a $p$-map operation $x\mapsto \pp{x}$ such that
$(A,\{-,-\}, \pp{(-)})$
is a restricted Lie algebra.
\item[(2)]
We call $A$ a {\it restricted Poisson algebra} if $A$ is a
weakly restricted Poisson algebra and the $p$-map $\pp{(-)}$
satisfies
\begin{equation}
\label{E0.1.1}\tag{E0.1.1}
\pp{(x^2)}=2x^p \pp{x}
\end{equation}
for all $x\in A$.
\end{enumerate}
\end{definition}

The formulation in \eqref{E0.1.1} is slightly simpler than the
original definition.
We will show that Definition \ref{xxdef0.1}(2) is
equivalent to \cite[Definition 1.8]{BK} in Lemma \ref{xxlem3.7}.
Generally it is not easy to prove basic properties for restricted
Poisson algebras. For example, it is not straightforward to show
that the tensor product preserves the restricted Poisson structure.
Different formulations are helpful in understanding and proving some
elementary properties.

Since there are several structures on a restricted Poisson algebra,
it is delicate to verify all compatibility conditions. There are not
many examples given in the literature. Our second
goal is to provide several canonical examples from different research
subjects. Restricted Poisson algebras can be viewed as a Poisson version
of restricted Lie algebras, so the first few examples come from
restricted (or modular) Lie theory. Let $L$ be a restricted
Lie algebra over $\K$. Then the trivial extension algebra $\K\oplus L$
(with $L^2=0$) is a restricted Poisson algebra. More naturally we have
the following.

\begin{theorem}[Theorem \ref{xxthm6.5}]
\label{xxthm0.2}
Let $L$ be a restricted Lie algebra over $\K$ and let $s(L)$ be
the $p$-truncated symmetric algebra. Then $s(L)$ admits a natural
restricted Poisson structure induced by the restricted Lie
structure of $L$.
\end{theorem}

To use ideas from Poisson geometry, it is a good idea to extend
the restricted Poisson structure to the symmetric algebra of a
restricted Lie algebra [Example \ref{xxex6.2}]. The following
result is slightly more general and useful in other setting.

\begin{theorem}[Theorem \ref{xxthm6.1}]
\label{xxthm0.3}
Let $T$ be an index set and
$A=\K[x_i \mid i\in T]$ be a polynomial Poisson algebra.
If, for each $i\in T$, there exists $\gamma(x_i)\in A$
such that $\ad_{x_i}^p=\ad_{\gamma(x_i)}$, then $A$ admits
a restricted Poisson structure $\pp{(-)}\colon A \to A$
such that $\pp{x_i}=\gamma(x_i)$ for all $i\in T$.
\end{theorem}

The next example comes from deformation theory,
which is also considered in \cite{BK}. See \eqref{E7.0.1}
for the definition of $M_n^p(f)$.

\begin{proposition}[Proposition \ref{xxpro7.1}]
\label{xxpro0.4}
Let $(A, \cdot, \{-,-\})$ be a Poisson algebra over $\K$
and let $(A[[t]], \ast)$ be a deformation
quantization of $A$.  If $M_n^p(f)=0$ for $1\le n\le p-2$
and $f^p$ is central in $A[[t]]$ for all $f\in A$,
then $A$ admits a restricted Poisson structure.
\end{proposition}

A Lie-Rinehart algebra is an algebraic counterpart of
a Lie algebroid, and appears naturally in the study of
Gerstenhaber algebras, Batalin-Vilkovisky algebras
and Maurer-Cartan algebras \cite{Hu1, Hu2}. In this paper,
we also study the relationship between restricted Poisson
algebras and restricted Lie-Rinehart algebras.

\begin{theorem}[Theorem \ref{xxthm8.2}]
\label{xxthm0.5}
Let $(A, \cdot, \{-,-\}, \pp{(-)})$ be a
restricted Poisson algebra. If the K\"ahler differential
$\Omega_{A/\K}$ is free over $A$,
then $(A, \Omega_{A/\K}, (-)^{[p]})$ is a restricted Lie-Rinehart algebra,
where the $p$-map of $\Omega_{A/\K}$ is determined by
\begin{align*}
&(x\d u)^{[p]}= x^p \d\pp{u}+(x \d u)^{p-1}(x)\d u,
\end{align*}
for all $x\d u \in \Omega_{A/\K}$.
\end{theorem}

The category of restricted Poisson
algebras is a symmetric monoidal category. In
particular, the tensor product of two restricted Poisson
algebras is again a restricted Poisson
algebra [Proposition \ref{xxpro9.2}].
Advances of algebra are tremendously benefited from geometric viewpoint
and methods and vice versa. Restricted Poisson algebras are, to some
extent, the algebraic counterpart of symplectic differential geometry
in positive characteristic. Following this idea, restricted Poisson-Lie
groups should correspond to restricted Poisson Hopf algebras
which connects both Poisson geometry in positive characteristic
and quantum groups at the root of unity. Hence, it is meaningful to
introduce the notion of a restricted Poisson Hopf algebra, see
Definition \ref{xxdef9.3}. One natural example of such an algebra
is given in Example \ref{xxex9.4}.

The paper is organized as follows.  Sections 1 and 2 contain
basic definitions about restricted Lie algebras and Poisson
algebras. In Section 3, we re-introduce the notion of a
restricted Poisson algebra. In Sections 4 to 7, we give several
natural examples. In Section 8, we prove Theorem \ref{xxthm0.5}.
The notion of a restricted Poisson Hopf algebra is introduced
in Section 9. The Appendix contains a combinatorial proof of
\eqref{E7.2.1} which is needed for Example \ref{xxex7.2}.

\section{Restricted Lie algebras}
\label{xxsec1}

We give a short review about restricted Lie algebras.

Lie algebras over a field of positive
characteristic often admit an additional structure involving a
so-called $p$-map. The Lie algebra together with a $p$-map is called
a \emph{restricted Lie algebra}, which was first introduced and
systematically studied by Jacobson \cite{J1, J2}. Let $L:=(L, [-, -])$
be a Lie algebra over $\K$. For convenience, for each $x\in L$, we
denote by $\ad_x\colon L\to L$ the adjoint representation given by
$\ad_x(y)=[x, y]$ for all $y\in L$. We recall the definition of a
restricted Lie algebra from \cite[Section 1]{J1}. As always, we
assume that $\K$ is of positive characteristic $p\geq 3$.

\begin{definition} \cite{J1}
\label{xxdef1.1}
A \emph{restricted Lie algebra} $(L, (-)^{[p]})$ over $\K$ is a Lie
algebra $L$ over $\K$ together with a \emph{$p$-map}
$(-)^{[p]}: x\mapsto x^{[p]}$
such that the following conditions hold:
\begin{enumerate}
\item[(1)]
$\ad_x^p=\ad_{x^{[p]}}$ for all $x\in L$;
\item[(2)]
$(\la x)^{[p]}=\la^p x^{[p]}$ for all $\la\in \K, x\in L$;
\item[(3)]
$(x+y)^{[p]}=x^{[p]}+y^{[p]}+\Lambda_p(x, y)$, where
$\Lambda_p(x, y)=\sum\limits_{i=1}^{p-1}\frac{s_i(x, y)}{i}$
for all $x, y\in L$ and $s_i(x, y)$ is the coefficient of
$t^{i-1}$ in the formal expression $\ad_{t x+y}^{p-1}(x)$.
\end{enumerate}
\end{definition}

For simplicity of notation, we write all multiple Lie brackets
with the notation
\begin{equation}
\label{E1.1.1}\tag{E1.1.1}
[x_1, [x_2, \cdots, [x_{n-1}, x_n]\cdots ]]
=:[x_1, x_2, \cdots, x_{n-1}, x_n],
\end{equation}
for $x_1, \cdots, x_n\in L$. Clearly,
$\ad_x^i(y)=[\underbrace{x, \cdots, x}_{i\ {\rm copies}}, y]$
for every $i$. Under this notation, we have
\begin{equation}
\label{E1.1.2}\tag{E1.1.2}
s_i(x, y)=\sum_{{x_k=x\ {\rm or}\ y}\atop{\#\{k\mid x_k=x\}=i-1}}
[x_1, \cdots, x_{p-2}, y, x],
\end{equation}
and, hence
\begin{equation}
\label{E1.1.3}\tag{E1.1.3}
\Lambda_p(x, y)=\sum_{{x_k=x\ {\rm or}\ y}\atop{x_{p-1}=y, x_p=x}}
\frac{1}{\#(x)} [x_1, \cdots, x_{p-1}, x_p].
\end{equation}

Note that $\Lambda_p(x,y)$ is denoted by $L(x,y)$ in \cite{BK}
and denoted by $\sigma(x,y)$ in \cite{Ho2}. Another way of
understanding $\Lambda_p(x, y)$ is to use the universal enveloping
algebra $\U(L)$ of the Lie algebra $L$. By \cite[Condition (3) on p. 559]{Ho2},
\begin{equation}
\label{E1.1.4}\tag{E1.1.4}
\Lambda_p(x, y)=(x+y)^{p}-x^{p}-y^{p}
\end{equation}
for all $x, y\in L\subset \U(L)$, where $(-)^{p}$ is the
multiplicative $p$-th power in $\U(L)$.

We give a well-known example which will be used later.

\begin{example}
\label{xxex1.2}
Let $A$ be an associative algebra over $\K$.
We denote by $A_L$ the induced Lie algebra with the bracket given
by $[x, y]:=xy-yx$, for all $x, y\in A$. Then $(A_L,(-)^p)$ is a
restricted Lie algebra, where $(-)^p$ is the Frobenius map given
by $x\mapsto x^p$.
\end{example}

In \cite[Theorem 11]{J2}, Jacobson gives a necessary and sufficient condition
in which an ordinary Lie algebra over $\K$ is restricted.

\begin{lemma}\cite[Theorem 11]{J2}
\label{xxlem1.3}
Let $L$ be a Lie algebra with a $\K$-basis $\{x_i\}_{i\in I}$
for some index set $I$.
Suppose that there exists an element $\gamma(x_i)\in L$ for
each $i\in I$ such that
$$\ad_{x_i}^p=\ad_{\gamma(x_i)}.$$
Then there exists a unique restricted structure on $L$ such that
$x_i^{[p]}=\gamma(x_i)$ for all $i\in I$.
\end{lemma}

\section{Poisson algebras and their enveloping algebras}
\label{xxsec2}

In this section we recall some definitions.
We refer to \cite{L-GPV} for some basics concerning Poisson algebras.

\begin{definition}
\cite[Definition 1.1]{L-GPV}
\label{xxdef2.1}
Let $A$ be a commutative algebra over $\K$. A {\it Poisson structure} on $A$
is a Lie bracket $\{-, -\}\colon A\ot A\to A$ such that the following
Leibniz rule holds
\begin{equation}
\label{E2.1.1}\tag{E2.1.1}
\{xy, z\}=x\{y, z\}+y\{x, z\}, \quad \forall \; x, y, z\in A.
\end{equation}
The algebra $A$ together with a Poisson structure is called a
\emph{Poisson algebra}.
\end{definition}

The Lie bracket $\{-, -\}$ (which replaces $[-,-]$ in the previous section)
is called the {\it Poisson bracket}, and the associative multiplication of
$A$ is sometimes denoted by $\cdot$. In this paper all Poisson algebras are
commutative as an associative algebra.

Recall that the K\" ahler differentials, denoted by $\Omega_{A/\K}$,  of a
commutative algebra $A$ over $\K$ is an $A$-module generated by
elements (or symbols) $\d x$ for all $x\in A$, and subject to the relations
\begin{align*}
\d(x+y)=\d x+\d y, \ \ \ \d(xy)=x\d y+y\d x,\ \ \  \d \la=0,
\end{align*}
where $x, y\in A, \la\in \K\subseteq A$.
When $(A,\{-,-\})$ is a Poisson algebra, the K\"ahler differentials
$\Omega_{A/\K}$ admits a Lie algebra structure with Lie bracket given by
\[ [x\d u, y\d v]=x\{u, y\} \d v+y\{x, v\} \d u+xy\d \{u, v\}\]
for all $x\d u, y\d v\in \Omega_{A/\K}$.
Moreover, $A$ is also a Lie module over $\Omega_{A/\K}$ with the action given
by
$(x\d u).a=x\{u, a\}$ for all $x\d u\in \Omega_{A/\K}, a\in A$.
In fact, the pair $(A, \Omega_{A/\K})$ is a Lie-Rinehart algebra
in the following sense.

\begin{definition}\cite[Definition 1.5]{Do}
\label{xxdef2.2}
A \emph{Lie-Rinehart algebra} over $A$ is a pair $(A, L)$, where
$A$ is a commutative associative algebra
over $\K$, $L$ is a Lie algebra equipped with the structure of an
$A$-module together with a map called \emph{anchor}
$$\al\colon L\to \Der_\K(A)$$
which is both an $A$-module and a Lie algebra homomorphism such that
\begin{equation}
\label{E2.2.1}\tag{E2.2.1}
[X, aY]=a[X, Y]+\al(X)(a)Y
\end{equation}
for all $a\in A$ and $X, Y\in L$.
\end{definition}

Note that, in the situation of Poisson algebra, the
anchor map $\alpha: \Omega_{A/\K} \to \Der(A)$ is given by
\begin{equation}
\label{E2.2.2}\tag{E2.2.2}
\al(x\d u)(z)=x\{u, z\}
\end{equation}
for all $x\d u\in \Omega_{A/\K}$ and $z\in A$.

Let $(A, L)$ be a Lie-Rinehart algebra. In \cite{Ri}, Rinehart
introduced the notion of universal enveloping algebra $\U(A, L)$
of $(A,L)$, which is an associative $\Bbbk$-algebra satisfying
the appropriate universal property, see \cite{Hu1} for more
details. We recall the definition next.

Denote by $A\rtimes L$ the semi-direct product of the Lie algebra
$L$ and the $L$-module $A$. More precisely, $A\rtimes L$ is the
direct sum of $A$ and $L$ as a vector space, and the Lie bracket
is given by
$$[(a, X), (b, Y)]=(X(b)-Y(a), [X, Y])$$
for all $(a, X), (b, Y)\in A\rtimes L$.
Let $(\U(A\rtimes L), \iota)$ be the universal enveloping algebra
of the Lie algebra $A\rtimes L$, where
$\iota\colon A\rtimes L \to \U(A\rtimes L)$ is the canonical embedding.
We consider the subalgebra $\U^+(A\rtimes L)$ (without unit)
generated by $A\rtimes L$.
Moreover, $A\rtimes L$ has the structure of an $A$-module via
$a(a', X)=(aa', aX)$ for all $a, a'\in A$ and $X\in L$.
The (universal) enveloping algebra $\U(A, L)$ associated to the
Lie-Rinehart algebra $(A,L)$ is defined to be the quotient
\[\U(A, L)=
\frac{\U^+(A\rtimes L)}{(\iota((a, 0))\iota((a', X))-\iota(a(a', X)))}.\]
Note that $(1_A,0)$ becomes the algebra identity of $\U(A, L)$.
There are two canonical maps
$$\iota_1\colon A\to \U(A, L), a\mapsto (a, 0)\quad {\text{ and}}\quad
\iota_2\colon L\to \U(A, L), X\mapsto (0, X).$$
Observe that $\iota_1$ is an algebra homomorphism and $\iota_2$
is a Lie algebra homomorphism. Moreover, we have the following relations
\[\iota_1(a)\iota_2(X)=\iota_2(aX),
\ {\rm and}\ \ [\iota_2(X), \iota_1(a)]=\iota_1(X(a))\]
for all $a\in A$ and $X\in L$.

As a consequence of \cite[Theorem 3.1]{Ri}, we have the following.

\begin{lemma}\label{xxlem2.3}
Let $(A,L)$ be a Lie-Rinehart algebra and $\U(A,L)$ the enveloping
algebra of $(A,L)$. If $L$ is a projective $A$-module, then the
Lie algebra homomorphism $\iota_2\colon L \to \U(A,L)$ is injective.
\end{lemma}

It is worth spending half page to re-state the above construction
for Poisson algebras since it is needed later. Denote by
$A\rtimes \Omega_{A/\K}$ the semidirect product of $A$ and $\Omega_{A/\K}$
with the Lie bracket given by
\[[(a, x\d u), (b, y\d v)]=(x\{u, b\}-y\{v, a\},
x\{u, y\} \d v+y\{x, v\} \d u+xy\d \{u, v\})\]
for $(a, x\d u), (b, y\d v)\in A\rtimes \Omega_{A/\K}$.
The Poisson enveloping algebra of $A$, denoted by $\P(A)$
(which is a new notation), is defined to be the enveloping
algebra of the Lie-Rinehart algebra $(A, \Omega_{A/\K})$, which
can be realized as an associated algebra
$$\P(A):=\U(A,\Omega_{A/\K})=\U^{+}(A\rtimes \Omega_{A/\K})/J,$$
where $\U(A\rtimes \Omega_{A/\K})$ is
the universal enveloping algebra of the Lie algebra $A\rtimes \Omega_{A/\K}$,
and $J$ is the ideal generated by
\begin{equation}
\label{E2.3.1}\tag{E2.3.1}
(a, 0)(b, x\d u)-(ab, ax\d u)
\end{equation}
for all
$a,b\in A, x\d u\in \Omega_{A/\K}$ \cite{MM, Ri}.
Here we have two maps
\[\iota_1\colon A\to A\rtimes \Omega_{A/\K} \to \P(A), \qquad \iota_1(a)=(a, 0)\]
and
\[\iota_2\colon \Omega_{A/\K} \to A\rtimes \Omega_{A/\K} \to \P(A),
\qquad \iota_2(x\d u)=(0, x\d u).\]
Then $\iota_1$ and $\iota_2$ are homomorphisms of associative algebras and
Lie algebras, respectively.
Moreover, we have
\begin{align}
\label{E2.3.2}\tag{E2.3.2}
&\iota_1(\{x, y\}) =[\iota_2(\d x), \iota_1(y)], \\
\label{E2.3.3}\tag{E2.3.3}
&\iota_2(\d(xy)) = \iota_1(x)\iota_2(\d y)+\iota_1(y)\iota_2(\d x)
\end{align}
for all $x, y\in A$.

If $\Omega_{A/\K}$ is a projective $A$-module,
then the canonical map $\iota_2\colon \Omega_{A/\K}\to \P(A)$ is injective
[Lemma \ref{xxlem2.3}].  It follows that
$\Omega_{A/\K}$ can be seen as a Lie subalgebra of $\P(A)$.

We now recall the definition of a free Poisson algebra, see
\cite[Section 3]{Sh}. Let $V$ be $\K$-vector space. Let $Lie(V)$ be the free
Lie algebra generated by $V$. The \emph{free Poisson algebra generated by $V$},
denoted by $FP(V)$, is the symmetric algebra over $Lie(V)$, namely
\begin{equation}
\label{E2.3.4}\tag{E2.3.4}
FP(V)=\K[Lie(V)].
\end{equation}

The following universal property is well-known \cite[Lemma 1, p. 312]{Sh}.

\begin{lemma}
\label{xxlem2.4}
Let $A$ be a Poisson algebra and $V$ be a vector space. Every
$\K$-linear map $g: V\to A$ extends uniquely to a Poisson algebra
morphism $G: FP(V)\to A$ such that $g$ factors through
$G$.
\end{lemma}

In \cite[Section 3]{Sh}, the notion of a free Poisson algebra is defined
by the universal property stated in Lemma \ref{xxlem2.4}, and then
Shestakov proved that the free Poisson algebra can be constructed by using
\eqref{E2.3.4} \cite[Lemma 1, p. 312]{Sh}. In \cite{Sh}, Shestakov
also considered the super (or ${\mathbb Z}_2$-graded) version of Poisson
algebras.

For each associative commutative algebra $A$, let $A^p$
denote the subalgebra generated by $\{f^p\mid f\in A\}$.
The free Poisson algebras have the following special property.

\begin{lemma}
\label{xxlem2.5} Let $A$ be a free Poisson algebra $FP(V)$.
\begin{enumerate}
\item[(1)]
$\Omega_{A/\K}$ is a free module over $A$. As a consequence,
the Lie algebra map $\iota_2\colon \Omega_{A/\K}\to \P(A)$ is injective.
\item[(2)]
The kernel of $\d: A\to \Omega_{A/\K}$ is $A^p$.
\end{enumerate}
\end{lemma}

\begin{proof}
(1) Since $A$ is a commutative polynomial ring,
$\Omega_{A/\K}$ is free over $A$. (The proof is omitted). The consequence
follows from Lemma \ref{xxlem2.3}.

(2) Check directly.
%
%
%
\end{proof}

Let $V$ be a $\K$-vector space. There are two gradings that
can naturally
be assigned to $FP(V)$. The first one is determined by
$$\deg_1 (x)=1, \quad \forall \; 0\neq x\in Lie(V).$$
Since $FP(V)$ is the symmetric algebra associated to $Lie(V)$,
the above extends to an $\N$-grading on $FP(V)$. Since the Lie bracket
$\{-,-\}$ has degree $-1$, the Poisson bracket on $FP(V)$ has degree
$-1$. Note that the multiplication on $FP(V)$ is homogeneous with
respect to $\deg_1$.

For the second grading, we assume that
$$\deg_2 (x)=1, \quad \forall \; 0\neq x\in V$$
and make the free Lie algebra $Lie(V)$ $\N$-graded (namely, $[-,-]$
is homogeneous of degree zero). Then we extend the
$\N$-grading to $FP(V)$ so that both the Poisson bracket
and the multiplication are homogeneous of degree zero.

Let $\{v_i\}_{i\in I}$ be a $\K$-basis of $V$ and
$\{x_j\}_{j\in J}$ a $\K$-basis of $Lie(V)$.
Let $A$ be the free Poisson algebra $FP(V)$ and let $A^{c}$ be
the $A^p$-submodule of $A$ generated by monomials $x_1^{i_1}
\cdots x_n^{i_n}$, for $x_1,\cdots,x_n\in Lie(V)$, which are
not in $A^p$.

Recall that
\begin{equation}
\label{E2.5.1}
\tag{E2.5.1}
\{f_1,f_2,\cdots, f_n\}:=
\{f_1, \{f_2,\cdots,\{f_{n-1},f_{n}\}\}
\end{equation}
for all $f_i\in A$.

\begin{lemma}
\label{xxlem2.6}
Let $A$ be a free Poisson algebra $FP(V)$.
\begin{enumerate}
\item[(1)]
Let $f_1,\cdots, f_n$ be polynomials in $v_i$
{\rm{(}}not $x_i${\rm{)}}. If $p$ does not divide $n-1$,
then $\{f_1,f_2,\cdots, f_n\}\in A^c$.
\item[(2)]
Let $f,g$ be polynomials in $v_i$. Then
$\Lambda_p(f,g)\in A^c$.
\item[(3)]
The following elements are in $A^c$ for any polynomials
in $f,g,h$ in $v_i$:
\begin{enumerate}
\item[(a)]
$\Lambda_{p}(f,g), \Lambda_p(f^2, g^2),\Lambda_p(f^2+g^2, 2fg).$
\item[(b)]
$\Lambda_{p}(fg,h), \Lambda_p((fg)^2, h^2),\Lambda_p((fg)^2+h^2, 2fgh).$
\item[(c)]
$\Lambda_p(fg,fh)$.
\end{enumerate}
\end{enumerate}
\end{lemma}

\begin{proof} (1) By linearity, we may assume that
all $f_s$ are monomials in $\{v_i\}\subseteq V$.
Then $\deg_1 f_s=\deg_2 f_s$ for $s=1,\cdots,n$.
Let $F:=\{f_1,f_2,\cdots, f_n\}$. Then
$$\deg_1 F=-n+1+\deg_2 F.$$
Since $p$ does not divide $n-1$, $p$ can not divide both
$\deg_1 F$ and $\deg_2 F$. This implies that $F\in A^c$.

(2) Note that $\Lambda_p(f,g)$ is a linear combination of
terms of the form \eqref{E2.5.1} when $n=p$ and
$f_i=f$ or $g$. By part (1), $\Lambda_p(f,g)\in A^c$.

(3) This is a special case of part (2) for different
choices of $f,g$.
\end{proof}

\section{Restricted Poisson algebras, Definition}
\label{xxsec3}

In this section we present a formulation of a restricted Poisson
algebra that is equivalent to \cite[Definition 1.8]{BK}.

Inspired by the notion of a restricted Lie algebra, we first introduce
the definition of a weakly restricted Poisson structure
over a field $\K$ of characteristic $p\geq 3$.

\begin{definition}
\label{xxdef3.1}
Let $(A, \cdot, \{-,-\})$ be a Poisson algebra. If $A$ admits
a $p$-map $\pp{(-)}\colon A\to A$ such that $(A, \{-, -\}, \pp{(-)})$
is a restricted Lie algebra, then $A$ is called a
\emph{weakly restricted Poisson algebra}.
\end{definition}

This definition requires no compatibility condition between
the $p$-map $\pp{(-)}$ and the multiplication $\cdot$.  We will see that
an additional requirement is very natural from a Lie algebraic
point of view.

\begin{lemma}
\label{xxlem3.2}
Let $(A, \cdot, \{-,-\})$ be a Poisson algebra and let $x, y\in A$.
\begin{enumerate}
\item[(1)]
If there exists $\widetilde{x}$ and $\widetilde{y}$ in $A$ such
that $\ad_x^p=\ad_{\widetilde x}$ and $\ad_y^p=\ad_{\widetilde y}$, then
$$\ad_{xy}^p=\ad_{x^p\widetilde{y}+y^p\widetilde{x}+\Phi_p(x, y)},$$
 where
\begin{equation}
\label{E3.2.1}\tag{E3.2.1}
\Phi_p(x, y)=(x^p+y^p)\Lambda_p(x, y)-
\frac{1}{2}(\Lambda_p(x^2, y^2)+\Lambda_p(x^2+y^2, 2xy)).
\end{equation}
In particular, $\ad_{x^2}^p=\ad_{2x^p\widetilde{x}}$.
\item[(2)]
If $(A, \cdot, \{-,-\})$ is a weakly restricted Poisson algebra,
then
\begin{equation}
\label{E3.2.2}\tag{E3.2.2}
\ad_{\pp{(xy)}}=\ad_{x^p\pp{y}+y^p\pp{x}+\Phi_p(x, y)}.
\end{equation}
In particular,
\begin{equation}
\label{E3.2.3}\tag{E3.2.3}
\ad_{\pp{(x^2)}}=\ad_{2x^p\pp{x}}.
\end{equation}
\end{enumerate}
\end{lemma}

\begin{proof} (1) We first prove the assertion when $x=y$.
By the Leibniz rule, we have $\ad_{(fg)}=f\ad_g+g\ad_f$
for any $f, g\in A$. Clearly,
$$\ad_{x^2}^p=(2x\ad_x)^p=(2x)^p(\ad_x)^p=2x^p\ad_x^p
=2x^p\ad_{\widetilde{x}}=\ad_{2x^p\widetilde{x}}.$$
In the general case, considering the
universal enveloping algebra of the Lie algebra $(A, \{-, -\})$ and
using \eqref{E1.1.4}, we get
$\ad_{\Lambda_p(f, g)}=\ad_{f+g}^p-\ad_{f}^p-\ad_{g}^p$ for any $f, g\in A$.
Therefore,
\begin{align*}
\ad_{x^p\widetilde{y}+y^p\widetilde{x}+\Phi(x, y)}
=&\ad_{x^p\widetilde{y}+y^p\widetilde{x}+
(x^p+y^p)\Lambda_p(x, y)-\frac{1}{2}(\Lambda_p(x^2, y^2)+\Lambda_p(x^2+y^2, 2xy))}\\
=& x^p\ad_y^p+y^p\ad_x^p+(x^p+y^p)(\ad_{x+y}^p-\ad_x^p-\ad_y^p)\\
&+\frac{1}{2}\left(\ad_{x^2}^p+\ad_{y^2}^p+ \ad_{2xy}^p-\ad_{(x+y)^2}^p\right)\\
= & x^p\ad_y^p+y^p\ad_x^p +(x^p+y^p)(\ad_{x+y}^p-\ad_x^p-\ad_y^p)\\
&+x^p\ad_x^p+y^p\ad_y^p+\ad_{xy}^p-(x+y)^p\ad_{x+y}^p\\
=& \ad_{xy}^p,
\end{align*}
which completes the proof.

(2) It is an immediate consequence of (1).
\end{proof}

Concerning the notation $\Phi_p$ in \eqref{E3.2.1}, we also have the
following characterization by considering the Poisson enveloping
algebra.

\begin{proposition}\label{xxpro3.3}
Let $A$ be a Poisson algebra and $\P(A)$ the Poisson enveloping
algebra of $A$. Then, for all $x, y\in A$, we have
\begin{align}
\label{E3.3.1}\tag{E3.3.1}
\iota_2(\d \Phi_p(x, y))=
(\iota_2(\d(xy)))^{p}-\iota_1(x^p)(\iota_2(\d y))^{p}
-\iota_1(y^p)(\iota_2(\d x))^{p}
\end{align}
\end{proposition}

\begin{proof}  By the definition of $\P(A)$, we have
\[(0, \d x^2)^p=(0,2x\d x)^p=((2x, 0)(0, \d x))^p
=(2x, 0)^p(0, \d x)^p=2(x^p, 0)(0, \d x)^p\]
and hence
\begin{equation}
\label{E3.3.2}\tag{E3.3.2}
(\iota_2(\d x^2))^p=2\iota_1(x^p)(\iota_2(\d x))^p
\end{equation}
for any $x\in A$.
It follows that the equation \eqref{E3.3.1} holds when $x=y$.

Considering the Frobenius map of $\P(A)$, we have
\begin{align*}
(\iota_2(\d(x+y)))^p = & (0, \d (x+y))^p=((0, \d x)+(0, \d y))^p\\
= & (0, \d x)^p+(0, \d y)^p+\Lambda_p((0, \d x), (0, \d y))\\
= & (\iota_2(\d x))^p+(\iota_2(\d y))^p+\iota_2(\d \Lambda_p(x, y))
\end{align*}
since $\iota_2$ is a homomorphism of Lie algebras. By the above computation
and \eqref{E3.3.2}, we have
\begin{align*}
(\iota_2(\d (x+y)^2))^p & = 2\iota_1((x+y)^p)(\iota_2(\d(x+y)))^p\\
& = 2\iota_1(x^p+y^p)((\iota_2(\d x))^p+(\iota_2(\d y))^p+\iota_2(\d \Lambda_p(x, y))).
\end{align*}
By a direct calculation and \eqref{E3.3.2},
\begin{align*}
(\iota_2(\d (x+y)^2))^p  =&(\iota_2(\d x^2+\d y^2+2\d(xy)))^p\\
 =& (\iota_2(\d x^2+\d y^2))^p+(\iota_2(2\d(xy)))^p+\iota_2(\d \Lambda_p(x^2+y^2, 2xy))\\
 =& (\iota_2(\d x^2))^p+(\iota_2(\d y^2))^p+\iota_2(\d \Lambda_p(x^2, y^2))\\
  & \quad +2(\iota_2(\d(xy)))^p+\iota_2(\d \Lambda_p(x^2+y^2, 2xy))\\
 =& 2\iota_1(x^p)(\iota_2(\d x))^p+2\iota_1(y^p)(\iota_2(\d y))^p
    +\iota_2(\d \Lambda_p(x^2, y^2))\\
  & \quad +2(\iota_2(\d(xy)))^p
	   +\iota_2(\d \Lambda_p(x^2+y^2, 2xy))
\end{align*}
Comparing the above two equations, we get
\begin{align*}
(\iota_2(\d (xy)))^p+&\frac{1}{2}(\iota_2(\d (\Lambda_p(x^2, y^2)
+\Lambda_p(x^2+y^2, 2xy))))\\
= & \iota_1(x^p)(\iota_2(\d y))^p+\iota_1(y^p)(\iota_2(\d x))^p
+\iota_1(x^p+y^p)\iota_2(\d \Lambda_p(x, y))\\
= & \iota_1(x^p)(\iota_2(\d y))^p+\iota_1(y^p)(\iota_2(\d x))^p
+\iota_2(\d((x^p+y^p)\Lambda_p(x, y))).
\end{align*}
Therefore,
\begin{align*}
\iota_2(\d \Phi_p(x, y)) &=\iota_2(\d ((x^p+y^p)\Lambda_p(x, y)-
\frac{1}{2}(\Lambda_p(x^2, y^2)+\Lambda_p(x^2+y^2, 2xy))))\\
&= (\iota_2(\d(xy)))^p-\iota_1(x^p)(\iota_2(\d y))^p-\iota_1(y^p)(\iota_2(\d x))^p.
\end{align*}
This finishes the proof.
\end{proof}

For a weakly restricted Poisson algebra, it is desired to consider some
compatibility between the $p$-map and the associative multiplication.
By removing $\ad$ from \eqref{E3.2.3} (which can be done in some
cases), we obtain \eqref{E3.4.1} below. Similarly, if we remove $\ad$ from
\eqref{E3.2.2}, we obtain \eqref{E3.5.1} below. Both Lemma \ref{xxlem3.2}
and Proposition \ref{xxpro3.3} suggest the following definition.
Following Lemma \ref{xxlem3.2}(2), condition \eqref{E3.4.1} is forced.

\begin{definition}
\label{xxdef3.4}
Let $(A, \cdot, \{-, -\}, (-)^{\{p\}})$ be a weakly restricted Poisson
algebra over $\K$. We call $A$ a \emph{restricted Poisson algebra}, if,
for every $x\in A$,
\begin{align}
\pp{(x^2)}=2x^p\pp{x}. \label{E3.4.1}\tag{E3.4.1}
\end{align}
In this case, the $p$-map $\pp{(-)}$ is a \emph{restricted Poisson
structure} on $A$.
\end{definition}

Next we give another description of condition \eqref{E3.4.1} which is
convenient for some computation.

\begin{proposition}
\label{xxpro3.5}
Let $A$ be a weakly restricted Poisson algebra.
\begin{enumerate}
\item[(1)]
Suppose  \eqref{E3.4.1} holds. Then
$\pp{(\la 1_A)}=0$, for all $\la\in \K$.
\item[(2)]
Equation \eqref{E3.4.1} holds for all $x\in A$ if and only if every
pair of elements $(x,y)$ in $A$ satisfies
\begin{align}\label{E3.5.1}\tag{E3.5.1}
\pp{(xy)}=x^p\pp{y}+y^p\pp{x}+\Phi_p(x, y).
\end{align}
As a consequence, $A$ is a restricted Poisson algebra if and
only if \eqref{E3.5.1} holds.
\item[(3)]
Suppose \eqref{E3.5.1} holds. Then
\begin{equation}
\label{E3.5.2}\tag{E3.5.2}
\pp{(x^n)}=n x^{(n-1)p} \pp{x}
\end{equation}
for all $n$. As a consequence, $\pp{(x^p)}=0$
for all $x\in A$.
\item[(4)]
If $\pp{(1_A)}=0$, then \eqref{E3.5.1} holds for pairs
$(x,\lambda 1_A)$ and $(\lambda 1_A,x)$ for all $x\in A$
and all $\lambda\in \K$.
\end{enumerate}
\end{proposition}

\begin{proof}
(1) Clearly, $\pp{1_A}=2\cdot 1_A^p \pp{1_A}$ and hence $\pp{1_A}=0$.
For every $\la\in \K$, $\pp{(\la 1_A)}=\la^p\pp{1_A}=0$.

(2) The `` if " part is trivial since $\Phi_p(x, x)=0$ for any $x\in A$.
Next, we show the `` only if " part. By \eqref{E3.4.1} and Definition
\ref{xxdef1.1}(3), we have
\begin{align*}
\pp{((x+y)^2)}=2(x+y)^p\pp{(x+y)}=2(x^p+y^p)(\pp{x}+\pp{y}+\Lambda_p(x, y))
\end{align*}
Since $(A, \{-, -\}, \pp{(-)})$ is a restricted Lie algebra, it follows
from Definition \ref{xxdef1.1}(2,3) that
\begin{align*}
\pp{((x+y)^2)}&=\pp{(x^2+y^2+2xy)}\\
&= \pp{(x^2+y^2)}+2^p\pp{(xy)}+\Lambda_p(x^2+y^2, 2xy)\\
&=\pp{(x^2)}+\pp{(y^2)}+\Lambda_p(x^2, y^2)+2^p\pp{(xy)}
   +\Lambda_p(x^2+y^2, 2xy)\\
&=2x^p\pp{x}+2y^p\pp{y}+\Lambda_p(x^2, y^2)+2\pp{(xy)}
   +\Lambda_p(x^2+y^2, 2xy)
\end{align*}
Comparing the above two equations and using $2\neq 0$, we obtain
equation\eqref{E3.5.1}.

(3) This follows by induction.

(4) First of all, $\pp{(\lambda 1_A)}=\lambda^p \pp{1_A}=0$
for all $\lambda\in \K$.
The assertion follows by the fact $\Phi_p(\lambda 1_A, x)=
\Phi_p(x,\lambda 1_A)=0$.
\end{proof}

\begin{remark}
\label{xxrem3.6}
Several remarks are collected below.
\begin{enumerate}
\item[(1)]
As in the paper \cite{BK}, we assume that $p\geq 3$.
So the polynomial $\Phi_p(x,y)$ in \eqref{E3.2.1}
is well-defined. When $p=3$, we have
\[\Phi_3(x, y)=x^2y\{y,y,x\}+xy^2\{x,x,y\}+xy\{x,y\}^2.\]
For any $p>3$, it is too long to write out all terms
like above.
\item[(2)]
Considering $\Phi_p(x,y)$ as an element in $FP(V)$,
where $V=\K x\oplus \K y$, it is homogeneous
of degree $p+1$ with respect to $\deg_2$ and homogeneous
of degree $2p$ with respect to $\deg_1$.
\item[(3)]
In \cite[Definition 1.8]{BK}, Bezrukavnikov-Kaledin defines
a {\it restricted Poisson algebra} as a weakly restricted
Poisson algebra $(A,\{-,-\}, \pp{(-)})$ such that the $p$-map
satisfies
\begin{equation}
\label{E3.6.1}\tag{E3.6.1}
\pp{(xy)}=x^p \pp{y}+y^p \pp{x}+P(x,y)
\end{equation}
for all $x,y\in A$. Here $P(x,y)$ is a canonical quantized
polynomial determined by \cite[{\rm{(1.3)}}]{BK}. We will show
that Equation \eqref{E3.6.1} is equivalent to \eqref{E3.5.1}.
\item[(4)]
The polynomial $P(x,y)$ is defined implicitly, but it
follows from \cite[{\rm{(1.3)}}]{BK}
that $P(x,x)=0$. Therefore a restricted Poisson algebra in the
sense of \cite[Definition 1.8]{BK} is a restricted Poisson
algebra in the sense of Definition \ref{xxdef3.4}.
\item[(5)]
There are other interpretations of $\Phi_p(x,y)$.
By using the equation
$$xy=\frac{1}{4} [(x+y)^2-(x-y)^2]$$
we obtain that
\begin{equation}
\label{E3.6.2}\tag{E3.6.2}
\pp{(xy)}=x^p \pp{y}+y^p \pp{x}+\Phi_p'(x,y)
\end{equation}
where
\begin{equation}
\label{E3.6.3}\tag{E3.6.3}
\Phi_p'(x,y)=\frac{1}{4}\Lambda_p((x+y)^2, -(x-y)^2)
+\frac{1}{2} ((x^p+y^p)\Lambda_p(x,y)-(x^p-y^p)\Lambda_p(x,-y)).
\end{equation}
One can show that $\Phi_p(x,y)=\Phi_p'(x,y)$ in the
free Poisson algebra generated by $x$ and $y$.
\item[(6)]
The following is clear by definition.
\begin{enumerate}
\item[(a)]
$\Lambda_p(x,y)=\Lambda_p(y,x)$ for all $x,y\in A$.
\item[(b)]
If $\{x,y\}=0$, then $\Lambda_p(x,y)=0$.
\item[(c)]
$\Phi_p(x,y)=\Phi_p(y,x)$ for all $x,y\in A$.
\item[(d)]
If $\{x,y\}=0$, then
$\Phi_p(x,y)=0$.
\end{enumerate}
\end{enumerate}
\end{remark}

\begin{lemma}
\label{xxlem3.7} Definitions of restricted
Poisson algebras in Definition {\rm{\ref{xxdef3.4}}}
and \cite[Definition 1.8]{BK} are
equivalent.
\end{lemma}

\begin{proof} Let $P(x,y)$ be the polynomial defined in
\cite[(1.3)]{BK}. By Proposition \ref{xxpro3.5}(2), it remains to
show that $P(x,y)=\Phi_p(x,y)$. Let $Lie(V)$ be the free
Lie algebra over a vector space $V$ and consider the tensor
(free) algebra $T(V)$ as a universal enveloping algebra over
$Lie(V)$. Then we have a Poincar{\'e}-Birkhoff-Witt filtration on
$T(V)$. The free quantized algebra $Q^{\bullet}(V)$ is the Rees
algebra associated to this filtration. By definition, for each $n$,
\[{\mathcal F}_n:=F_n T(V)= \Bbbk \oplus L^\bullet(V)
\oplus (L^\bullet(V))^2 \oplus \cdots \oplus (L^\bullet(V))^{n}.\]
We are omitting the symbol $h$ which represents the natural
embedding $h: {\mathcal F}_{\bullet}\to {\mathcal F}_{\bullet+1}$
in the Rees ring. Taking $V=\Bbbk x\oplus \Bbbk y$, we have the
following computation inside the Rees ring
\begin{align*}
(x^p+y^p)^2+&(\Lambda_p(x, y))^2+\Lambda_p(x, y)(x^p+y^p)
       +(x^p+y^p)\Lambda_p(x, y)\\
=&(x^p+y^p+\Lambda_p(x, y))^2\\
=& (x+y)^{2p}\\
=&(x^2+y^2+xy+yx)^p\\
=& (x^2+y^2)^{p}+(xy+yx)^{p}+\Lambda_p(x^2+y^2, xy+yx)\\
=& x^{2p}+y^{2p}+\Lambda_p(x^2, y^2)+(xy)^{p}+(yx)^p+\Lambda_p(xy, yx)\\
 &\qquad\qquad\qquad +\Lambda_p(x^2+y^2, xy+yx),
\end{align*}
and hence
\begin{align*}
(xy)^p+(yx)^p&-x^py^p-y^px^p\\
= &\Lambda_p(x, y)(x^p+y^p)+(x^p+y^p)\Lambda_p(x, y)+(\Lambda_p(x, y))^2\\
& \qquad -\Lambda_p(x^2, y^2)-\Lambda_p(xy, yx)-\Lambda_p(x^2+y^2, xy+yx).
\end{align*}
On the other hand,
\begin{align*}
[x, y]^p=(xy-yx)^p=(xy)^p-(yx)^p+\Lambda_p(xy, -yx).
\end{align*}
So we have
\begin{align*}
2P(x,y)&= 2((xy)^p-x^py^p)\\
= & \Lambda_p(x, y)(x^p+y^p)+(x^p+y^p)\Lambda_p(x, y)
       -\Lambda_p(x^2, y^2)-\Lambda_p(x^2+y^2, xy+yx)\\
&\qquad +(\Lambda_p(x, y))^2-\Lambda_p(xy, yx)
       -\Lambda_p(xy, -yx)+[x, y]^p-[x^p, y^p].
\end{align*}
In fact, it is easily seen that $(\Lambda_p(x, y))^2\in {\mathcal F}_2,
[x, y]^p\in {\mathcal F}_p$.
On the other hand,
\[[x^p, y^p]=\ad_x^p(y^p))=-\ad_x^{p-1}(\ad_y^p(x))\in {\mathcal F}_1,\]
where $\ad_x(y)=[x, y]$. By the equation (E1.1.3),  we have
\[\Lambda_p(xy, yx)=\sum_{x_k=xy\ {\rm or}\ yx}\dfrac{1}{\#(xy)}
\ad_{x_1}\cdots\ad_{x_{p-2}}([yx, xy]).\]
Since $[yx, xy]=[yx, yx+[x, y]]=[yx, [x, y]]\in {\mathcal F}_2$, we have
$\Lambda_p(xy, yx)\in {\mathcal F}_p$.
Similarly, $\Lambda_p(xy, -yx)\in {\mathcal F}_p$.
By definition \cite[(1.3)]{BK}, $P(x,y)$ is homogeneous of
degree $p+1$.
Therefore, after removing lower degree components,
\[2P(x, y)=\Lambda_p(x, y)(x^p+y^p)+(x^p+y^p)\Lambda_p(x, y)
-\Lambda_p(x^2, y^2)-\Lambda_p(x^2+y^2, xy+yx).\]
Since the multiplication is commutative in a Poisson algebra, we have
\begin{align*}
 P(x, y)=  (x^p+y^p)\Lambda_p(x, y)-\dfrac{1}{2}
(\Lambda_p(x^2, y^2)+\Lambda_p(x^2+y^2, 2xy))
=\Phi_p(x, y).
\end{align*}
\end{proof}

\section{Elementary properties and examples}
\label{xxsec4}

We start with something obvious.

\begin{definition}
\label{xxdef4.1}
Let $(A, \cdot, \{-, -\}, \pp{(-)})$ be a restricted Poisson algebra.
A Poisson ideal $I$ of $A$ is said to be \emph{restricted}, if
$\pp{x}\in I$ for any $x\in I$.
\end{definition}

The proofs of the following three assertions are easy
and omitted.

\begin{lemma}
\label{xxlem4.2}
Let $A$ be a restricted Poisson algebra. Suppose that $I$ is a Poisson
ideal of $A$ that is generated by $\{x_i\mid i\in S\}$ as an ideal of
the commutative ring $A$. If $\pp{x_i}\in I$ for any $i\in S$, then
$I$ is a restricted Poisson ideal.
\end{lemma}


\begin{proposition}
\label{xxpro4.3}
Let $A$ be a restricted Poisson algebra and $I$ a restricted Poisson
ideal of $A$. Then the quotient Poisson algebra $A/I$ is a restricted
Poisson algebra.
\end{proposition}


As a consequence, we have.

\begin{corollary}
\label{xxcor4.4}
Let $f\colon A \to A'$ be a homomorphism of restricted Poisson algebras.
Then $\Ker f$ is a restricted Poisson ideal of $A$.
\end{corollary}

Let $A^p$ be the subalgebra of $A$ generated by $\{f^p\mid f\in A\}$ -- the
image of the Frobenius map.

\begin{lemma}
\label{xxlem4.5}
Let $A$ be a Poisson algebra and $f, g, h\in A$. Then the following hold:
\begin{enumerate}
\item[(1)]
$f^p\Phi_p(g, h)-\Phi_p(fg, h)+\Phi_p(f, gh)-h^p\Phi_p(f, g)=0$.
\item[(2)]
If $f$ is in the Poisson center of $A$, then
$$f^p\Phi_p(g, h)=\Phi_p(fg, h)=\Phi_p(g, fh).$$
\item[(3)]
$\Phi_p(f, g+h)-\Phi_p(f, g)-\Phi_p(f, h)=\Lambda_p(fg, fh)-f^p\Lambda_p(g, h)$.
\end{enumerate}
\end{lemma}

\begin{proof} It is clear that (2) is a consequence of (1). It suffices to show
assertions (1) and (3) for the free Poisson algebra $FP(A)$ since there is a
surjective Poisson algebra map $FP(A)\to A$ [Lemma \ref{xxlem2.4}].
So the hypothesis becomes that $f,g,h$ are in a $\K$-space $V$
sitting inside a free Poisson algebra $FP(V)$.

When $A$ is a free Poisson algebra $FP(V)$, by Lemma \ref{xxlem2.5}(1),
$\iota_2$ is injective. It follows from Lemma \ref{xxlem2.5}(2) that
\begin{enumerate}
\item[(a)]
the kernel of the map
$$A\xr{\d} \Omega_{A/\K}\xr{\iota_2} \P(A)$$
is $A^p$.
\end{enumerate}
Let $\{v_i\}_{i\in S}$ be a basis of the $V$. Let $A^c$ be the
$A^p$-submodule of $A=FP(V)$ defined before Lemma \ref{xxlem2.6}.
Then
\begin{enumerate}
\item[(b)]
$A^c\cap A^p=\{0\}$ and $\Lambda_p(x,y)\in A^c$
for all $x,y\in \K[V]$ by Lemma \ref{xxlem2.6}(2).
\end{enumerate}

Now we prove (1) and (3) under conditions (a) and (b).

(1)
For all $f, u\in A$, $\d (f^p u)=f^p\d u$ and
$\iota_2(\d (f^pu))=(f^p, 0)(0, \d u)\in \P(A)$.
By Proposition \ref{xxpro3.3}, we have
\begin{align*}
\iota_2(\d (f^p\Phi_p(g, h)))
& = (f^p, 0)(0, \d(gh))^p-(f^p g^p, 0)(0, \d h)^p-(f^ph^p, 0)(0, \d g)^p,\\
\iota_2(\d \Phi_p(fg, h))
& = (0, \d (fgh))^p-((fg)^p, 0)(0, \d h)^p-(h^p, 0)(0, \d (fg))^p,\\
\iota_2(\d \Phi_p(f, gh))
& = (0, \d (fgh))^p-(f^p, 0)(0, \d (gh))^p-((gh)^p, 0)(0, \d f)^p,\\
\iota_2(\d (h^p\Phi_p(f, g))
&=(h^p, 0)(0, \d(fg))^p-(h^pf^p, 0)(0, \d g)^p-(h^pg^p, 0)(0, \d f)^p.
\end{align*}
for all $f, g, h\in V$. It follows that
\[\iota_2(\d (f^p\Phi_p(g, h)-\Phi_p(fg, h)+\Phi_p(f, gh)-\Phi_p(f, g)h^p))=0.\]
By condition (a), we get $X:=f^p\Phi_p(g, h)-\Phi_p(fg, h)
+\Phi_p(f, gh)-h^p \Phi_p(f, g)\in A^p$.
By definition, $X$ is in the $A^p$-submodule generated by $\Lambda_p(x,y)$
for all $x,y\in A$, or in $A^c$ as given in condition (b).
But $A^p\cap A^c=\{0\}$ by condition (b), we obtain that
$X=0$ and that the desired identity holds.

(3) The proof of part (3) is similar to the proof of (1)
and is omitted.
\end{proof}

\begin{proposition}
\label{xxpro4.6}
Let $A$ be a weakly restricted Poisson algebra.
\begin{enumerate}
\item[(1)]
If $(x,y)$ satisfies \eqref{E3.5.1}, then so do $(x,\lambda y)$
and $(\lambda x, y)$ for all $\lambda\in \K$.
\item[(2)]
Let $f,g,h\in A$. Suppose that $(f,g)$ and $(g,h)$
satisfy \eqref{E3.5.1}.
Then $(fg,h)$ satisfies \eqref{E3.5.1} if and
only if $(f,gh)$ does.
\item[(3)]
If $(f,g)$ and $(f,h)$ satisfies \eqref{E3.5.1}, then so does
$(f, g+h)$.
\item[(3')]
If $(g,f)$ and $(h,f)$ satisfies \eqref{E3.5.1}, then so does
$(g+h,f)$.
\item[(4)]
Fix an $x\in A$ and let $R_x$ be the set of $y\in A$ such that
$(x,y)$ satisfies \eqref{E3.5.1}. Then $R_x$ is a $\Bbbk$-subspace
of $A$.
\item[(4')]
Fix an $x\in A$ and let $L_x$ be the set of $y\in A$ such that
$(y,x)$ satisfies \eqref{E3.5.1}. Then $L_x$ is a $\Bbbk$-subspace
of $A$.
\end{enumerate}
\end{proposition}

\begin{proof}
(1) Assuming \eqref{E3.5.1}  for $(x,y)$, we have
$$\begin{aligned}
\pp{( x \lambda y)} &= \pp{\lambda(xy)}=\lambda^{p} \pp{(xy)}\\
&=\lambda^{p}( x^p \pp{y}+y^p\pp{x} +\Phi_p(x,y))\\
&= x^p \pp{(\lambda y)} +(\lambda y)^p \pp{x}+\lambda^{p}\Phi_p(x,y))\\
&= x^p \pp{(\lambda y)} +(\lambda y)^p \pp{x}+\Phi_p(x,\lambda y)),
\end{aligned}
$$
where the last equation is Lemma \ref{xxlem4.5}(2).
So $(x,\lambda y)$ satisfies \eqref{E3.5.1}. Similarly for
$(\lambda x,y)$.

(2) By symmetry, we only prove one implication and assume
that $(fg, h)$ satisfies \eqref{E3.5.1}. We
show next that $(f,gh)$ satisfies \eqref{E3.5.1}:
$$\begin{aligned}
\pp{(f(gh))} &=\pp{((fg)h)}=(fg)^p \pp{h}+h^p \pp{(fg)}
+\Phi_p(fg, h)\\
&=(fg)^p \pp{h}+h^p(f^p\pp{g}+g^p\pp{f}+\Phi_p(f,g))
+\Phi_p(fg,h)\\
&=f^pg^p\pp{h}+f^ph^p\pp{g}+g^ph^p\pp{f}
+\Phi_p(fg,h)+h^p\Phi_p(f,g)\\
&=f^pg^p\pp{h}+f^ph^p\pp{g}+g^ph^p\pp{f}\\
&\qquad\qquad\quad +f^p\Phi_p(g,h)+\Phi_p(fg,h)\qquad
{\text{by Lemma \ref{xxlem4.5}(1)}}\\
&=f^p(g^p\pp{h}+h^p\pp{g}+\Phi_p(g,h))+(gh)^p\pp{f}+\Phi_p(f,gh)\\
&=f^p \pp{(gh)}+(gh)^p\pp{f}+\Phi_p(f,gh).
\end{aligned}
$$

(3) Assume $(f,g)$ and $(f,h)$ satisfies \eqref{E3.5.1}. Then
\begin{align*}
\pp{(f(g+h))}=&\pp{(fg+fh)}\\
=& \pp{(fg)}+\pp{(fh)}+\Lambda_p(fg, fh)\\
=& f^p\pp{g}+g^p\pp{x}+\Phi_p(f, g)+x^p\pp{h}+h^p\pp{f}\\
& \qquad\qquad\qquad +\Phi_p(f, h)+\Lambda_p(fg, fh)\\
=& f^p(\pp{g}+\pp{h}+\Lambda_p(g, h))+(g+h)^p\pp{f}+\Phi_p(f, g+h)\\
=& f^p\pp{(g+h)}+(g+h)^p \pp{f}+\Phi_p(f, g+h), 
\end{align*}
where the second last equality is deduced from Lemma \ref{xxlem4.5}(3).
So $(f,g+h)$ satisfies \eqref{E3.5.1}.

(3') is equivalent to (3).

(4) Let $$R_x=\{y\in A \mid {\text{\eqref{E3.5.1}
holds for the pair $(x,y)$}}\}.$$
By Proposition \ref{xxpro4.6}(1), we have
\begin{enumerate}
\item[(i)]
if $y\in R_x$, then so is $\lambda y$ for all $\lambda\in \Bbbk$.
\end{enumerate}
By Proposition \ref{xxpro4.6}(3),
\begin{enumerate}
\item[(ii)]
if $g,h\in R_x$, then so is $g+h$.
\end{enumerate}
By (i) and (ii) above, $R_x$ is a $\Bbbk$-subspace of $A$.

(4') This is true because $L_x=R_x$.
\end{proof}

The following result will be used several times.

\begin{theorem}
\label{xxthm4.7}
Let $A$ be a weakly restricted Poisson algebra.
Let ${\bf b}:=\{b_i\}_{i\in S}$ be a $\Bbbk$-basis of $A$. If \eqref{E3.5.1}
holds for every pair $(x,y)\subseteq {\bf b}$, then $A$ is a restricted
Poisson algebra.
\end{theorem}

\begin{proof}
We need to show that \eqref{E3.5.1} holds for all $x,y\in A$.
First we fix any $x\in {\bf b}$ and let
$$R_x=\{y\in A \mid {\text{\eqref{E3.5.1}
holds for the pair $(x,y)$}}\}.$$
By Proposition \ref{xxpro4.6}(4), $R_x$ is a $\Bbbk$-subspace of
$A$. By hypothesis, we see that ${\bf b}\subseteq R_x$.
Since $\bf b$ is a basis of $A$, $R_x=A$.

Next we fix $y\in A$ and consider
$$L_y=\{x\in A \mid {\text{\eqref{E3.5.1}
holds for the pairs $(x,y)$}}\}.$$
Similarly, by Proposition \ref{xxpro4.6}(4'),
$L_y$ is a $\Bbbk$-subspace. It contains $\bf b$ because
$R_x=A$ for all $x\in {\bf b}$ (see the first paragraph).
Hence, $L_y=A$. This means that
\eqref{E3.5.1} holds for all pairs $(x,y)$ in $A$. Therefore
$A$ is a restricted Poisson algebra.
\end{proof}

One of the main goals of this paper is to provide some interesting
examples of restricted Poisson algebras. In the rest of this section
we give some elementary (but nontrivial) examples. We would like to
give a gentle warning before the examples. We have checked that
all $p$-maps given below satisfy \eqref{E3.5.1}, however our
proofs are tedious computations and therefore omitted. On the other
hand, since the $p$-maps are explicitly expressed by partial
derivatives, one can verify the assertions with enough patient.
More sophisticated examples are given in later sections.

\begin{example}
\label{xxex4.8}
Let $A=\K[x, y]$ be a polynomial algebra in two variables $x, y$,
where the (classical) Poisson bracket is given by
\begin{equation}
\label{E4.8.1}\tag{E4.8.1}
\{f, g\}=f_xg_y-f_yg_x.
\end{equation}
for all $f, g\in A$, and $f_x, f_y$ are the partial derivative
of $f$ with respect to the variables $x$ and $y$, respectively.
(The bracket defined in \eqref{E4.8.1} was the original Poisson
bracket studied by many people including Poisson \cite{Po} when
$\Bbbk={\mathbb R}$.)
\begin{enumerate}
\item[(1)]
Let $\K$ be a base field of characteristic $3$. For every
$f\in A$, we define
\begin{equation}
\label{E4.8.2}\tag{E4.8.2}
f^{\{3\}}=f_x^2f_{yy}+f_y^2f_{xx}+f_xf_yf_{xy},
\end{equation}
where $f_{xx}, f_{yy}$ and $f_{xy}$ are the second order
partial derivatives of $f$. Then $(A, \cdot, \{-, -\}, (-)^{\{3\}})$
is a restricted Poisson algebra.
\item[(2)]
Let $\K$ be a base field of characteristic $5$. For every $f\in A$, define
\begin{align}
f^{\{5\}}
=&f_1^4f_{2222}+f_1^3f_2f_{1222}+f_1^2f_2^2f_{1122}
  +f_1f_2^3f_{1112}+f_2^4 f_{1111}\notag\\
&+f_{12}(f_1^3f_{222}-f_1^2f_2f_{122}-f_1f_2^2f_{112}+f_2^3f_{111})\notag\\
&-f_1f_{22}(f_1^2f_{122}-2f_1f_2f_{112}+f_2^2f_{111})\label{E4.8.3}\tag{E4.8.3}\\
&-f_2f_{11}(f_2^2f_{112}-2f_2f_1f_{122}+f_1^2f_{222})\notag\\
&+2(f_{12}^2-f_{11}f_{22})(f_1^2f_{22}-2f_1f_2f_{12}+f_2^2f_{11}),\notag
\end{align}
where $f_{i_1i_2\cdots i_k}$ denotes the $k$-th order partial
derivative of $f$ with respect to the variables
$x_{i_1}, x_{i_2}, \cdots, x_{i_k}$. Then
$(A, \cdot, \{-, -\}, (-)^{\{5\}})$ is a restricted Poisson algebra.
\end{enumerate}

See Example \ref{xxex7.2} for general $p$.
It would be interesting to understand
the meaning of \eqref{E4.8.2} and \eqref{E4.8.3}
and to find its connection with other subjects.
\end{example}

%

The next two are slight generalizations of the previous example.

\begin{example}
\label{xxex4.9}
Suppose ${\text{char}}\; \K=3$ and let
$A=\K[x, y]$ be a polynomial Poisson algebra in two variables
$x, y$, where the Poisson bracket is given by
\[\{f, g\}=\varphi(f_xg_y-f_yg_x),\]
and $\varphi=\la x+\mu y+\nu$, $\la, \mu, \nu\in \K$.
For every $f\in A$, we define
\begin{equation}
\label{E4.9.1}\tag{E4.9.1}
\pk{f}=\la \varphi f_xf_y^2+\mu\varphi f_x^2f_y+
\varphi^2(f_x^2f_{yy}+f_y^2f_{xx}+f_xf_yf_{xy})+
\la^2yf_y^3+\mu^2xf_x^3.
\end{equation}
Then $(A, \cdot, \{-, -\}, \pk{(-)})$ is a restricted Poisson algebra.
\end{example}

\begin{example}
\label{xxex4.10}
Suppose ${\text{char}}\; \K=3$ and let
$A=\K[x_1, x_2, \cdots, x_n]$ be a Poisson algebra,
where the Lie bracket is given by
$\{x_i, x_j\}=2c_{ij}\in \K $ with $c_{ij}+c_{ji}=0$
for $1\le i, j\le n$.
Clearly, $\{f, g\}=\sum_{1\le i, j\le n}c_{ij}(f_ig_j-f_jg_i)$
for $f, g\in A$,
where $f_i$ denotes the partial derivative of $f$ with
respect to the variable $x_i$ for $i=1, 2, \cdots, n$.
Then $A$ is a restricted Poisson algebra with the $p$-map given by
\[f^{\{3\}}=\Sum_{1\le i,j,k,l\le n}c_{ij}c_{kl}f_if_kf_{jl}\]
for any $f\in A$, where $f_{jl}$ is the second partial derivation
of $f$ with respect to the variables $x_j$ and $x_l$.
\end{example}

\section{Existence and uniqueness of restricted structures}
\label{xxsec5}
By Lemma \ref{xxlem3.2}(2), a weakly restricted Poisson structure
on a Poisson algebra is very close to a restricted Poisson structure
(up to a factor in the Poisson center). In this section, we study
the existence and uniqueness of (weakly) restricted Poisson structure.
First we consider the trivial extension.

\begin{lemma}
\label{xxlem5.1}
Let $A$ be a Poisson algebra and $A=\Bbbk 1_A \oplus \fm$ as a
Lie algebra decomposition.
\begin{enumerate}
\item[(1)]
If $x\mapsto \pp{x}$ {is a  restriction} $p$-map of the Lie
algebra $\fm$, then it can naturally be extended to $A$ by
defining $\pp{1_A}=0$. As a consequence, $A$ is a
weakly restricted Poisson algebra.
\item[(2)]
If, further, the $p$-map on $\fm$ satisfies  \eqref{E3.4.1},
then so does the extended $p$-map on $A$. In this case,
$A$ is a restricted Poisson algebra.
\end{enumerate}
\end{lemma}

\begin{proof}
(1) This follows from Lemma \ref{xxlem1.3}. For all
$\lambda\in \Bbbk$ and $x\in \fm$, the $p$-map is defined
by $\pp{(\lambda 1_A+x)}=\pp{x}$.

(2) We check \eqref{E3.4.1} for elements in $A$ as follows:
$$\begin{aligned}
\pp{((\lambda 1_A+x)^2)}&= \pp{(\lambda^2 1_A+ 2\lambda x+x^2)}\\
&=\pp{(2\lambda x+x^2)}=\pp{(2\lambda x)}+\pp{(x^2)}\\
&=2\lambda^p \pp{x}+2 x^p\pp{x}=2(\lambda 1_A+x)^p \pp{x}\\
&=2(\lambda 1_A+x)^p \pp{(\lambda 1_A+x)}.
\end{aligned}
$$
Therefore $A$ is a restricted Poisson algebra.
\end{proof}

The following example is immediate.

\begin{example}
\label{xxex5.2}
\begin{enumerate}
\item[(1)]
Let $L$ be a restricted Lie algebra and let $A=\Bbbk 1_A\oplus L$
where associate product $L^2=0$. Then $A$ is a Poisson algebra
in the obvious way. Both sides of \eqref{E3.4.1} are zero for
elements in $L$ (since $L^2=0$). By Lemma \ref{xxlem5.1}(2),
$A$ is a restricted Poisson algebra.
\item[(2)]
Considering a special case when
$L=\Bbbk x+\Bbbk y$ is a solvable Lie algebra
with $[x, y]=x$. For $f=\lambda_1x+\lambda_2y\in L$,
we define the $p$-map by
\[\pp{f}=\lambda_2^{p-1}(\lambda_1x+\lambda_2y).\]
It is straightforward to check that $(L, \pp{(-)})$
is a restricted Lie algebra.
Let $A=\Bbbk 1_A \oplus L$. Then, by part (1), $A$ is a
restricted Poisson algebra. As a commutative algebra,
$A=\frac{\K[x, y]}{(x^2, xy, y^2)}$ with $\K$-linear basis
$\{1, x, y\}$. The Poisson bracket is given by $\{x, y\}=x$.
\end{enumerate}
\end{example}

Let $L$ be a restricted Lie algebra. It is well known
that the $p$-map of $L$ is unique up to a semilinear map
from $L$ to $Z(L)$, where $Z(L)$ is the center of $L$.
Recall that a semilinear map $\gamma\colon L\to Z(L)$
means that for any $x, y\in A, \la \in \K$,
\begin{align*}
\gamma(x+y) & = \gamma(x)+\gamma(y),\\
\gamma(\la x) &= \la^p \gamma(x).
\end{align*}

The following lemma is well-known and easy to prove.

\begin{lemma}\label{xxlem5.3}
Let $(L, (-)^{[p]})$ be a restricted Lie algebra.
\begin{enumerate}
\item[(1)]
Let $\pp{(-)}$ be another restricted Lie structure
on $L$. Then there is a maps $\gamma: L\to Z(L)$ such that
$\pp{(-)}=(-)^{[p]}+\gamma$.
\item[(2)]
Let $\gamma$ be a map from $L$ to $Z(L)$. Then
$(-)^{[p]}+\gamma$ is a restricted Lie structure on
$L$ if and only if $\gamma$ is a semilinear map from
$L$ to $Z(L)$.
\end{enumerate}
\end{lemma}

Let $A$ be a Poisson algebra over $\K$ and $Z(A)$ the
center of $A$. Observe that $Z(A)$
is a left $A$-module with the action given by
$$A\times Z(A) \to Z(A),\quad (a, z)\mapsto a^pz.$$
A semilinear map $\psi\colon A\to Z(A)$ is called
a \emph{Frobenius derivation} of $A$ with the
value in $Z(A)$ provided that
$\psi(ab)=a^p\psi(b)+b^p\psi(a)$ for any $a, b\in A$.
For example, if $\psi_0\colon A\to A$ is
a derivation, then $\psi\colon A\to Z(A)$, defined by
$\psi(a)=(\psi_0(a))^p$ for all $a\in A$,
is a Frobenius derivation of $A$ with the value in $Z(A)$.

By Lemma \ref{xxlem5.3}(1), any two restricted Poisson structures
on $A$ differ by a semilinear map $\gamma$ which appears in the
next proposition.

\begin{proposition}
\label{xxpro5.4}
Let $(A, \cdot, \{-,-\}, (-)^{\{p\}})$ be a restricted
Poisson algebra and $\gamma$ a map from $A$
to itself. Then the map $(-)^{\{p\}}+\gamma$ is a restricted
Poisson structure if and only if $\gamma$ is a Frobenius
derivation of $A$ with value in $Z(A)$.
\end{proposition}

\begin{proof}
Let $(-)^{\{p\}_1}\colon A\to A$ be another $p$-map such that
$(A, \cdot, \{-, -\}, (-)^{\{p\}_1})$ is also a
restricted Poisson algebra.
Since $(-)^{\{p\}_1}$ and $\pp{(-)}$ are restricted structures
on Lie algebra $(A, \{-,-\})$,
$\gamma=(-)^{\{p\}_1}-\pp{(-)}$ is a semilinear map from
$A$ to $Z(A)$ by Lemma \ref{xxlem5.3}.  Moreover, for any $x, y\in A$,
$(xy)^{\{p\}_1}=x^py^{\{p\}_1}+y^px^{\{p\}_1}+\Phi_p(x, y)$,
and
\begin{align*}
\gamma(xy)= & (xy)^{\{p\}_1}-\pp{(xy)}\\
= & x^p (y^{\{p\}_1}-\pp{y})+y^p (x^{\{p\}_1}-\pp{x})\\
= & x^p\gamma(y)+y^p\gamma(x)\\
\end{align*}
It follows that $\gamma$ is a Frobenius derivation of $A$
with values in $Z(A)$.

Conversely, it follows from Lemma \ref{xxlem5.3} that
the map $\pp{(-)}+\gamma$ is also a restricted Lie
structure on $(A, \{-, -\})$, since $\gamma$ is a
semilinear map from $A$ to $Z(A)$ and $\pp{(-)}$ is a
$p$-map of Lie algebra $(A, \{-, -\})$.
Moreover, for any $x, y\in A$,
\[\pp{(xy)}+\gamma(xy)
=x^p (\pp{y}+\gamma(y))+y^p(\pp{x}+\gamma(x))+\Phi_p(x, y).\]
It follows that the Poisson algebra $A$
together with the map $\pp{(-)}+\gamma$ is a restricted structure.
\end{proof}

By Proposition \ref{xxpro5.4}, the $p$-map of a restricted Poisson algebra
is unique up to Frobenius derivations.

\begin{remark}
\label{xxrem5.5} Let $(A, \cdot, \{-,-\}, \pp{(-)})$ be a restricted
Poisson algebra and let $\gamma: A\to Z(A)$ be a semilinear map.
Suppose that $\gamma$ is not a Frobenius derivation (which is possible
for many $A$) and defines a new $p$-map $\pp{(-)'}=\pp{(-)}+\gamma$.
Then by Proposition \ref{xxpro5.4}, $(A, \cdot, \{-,-\}, \pp{(-)'})$
is not a restricted Poisson algebra, but it is still a weakly
restricted Poisson algebra by Lemma \ref{xxlem5.3}(2).
\end{remark}

\section{Restricted Poisson algebras from restricted Lie algebras}
\label{xxsec6}

We start with a general result.

\begin{theorem}
\label{xxthm6.1}
Let $A=\K[x_i \mid i\in T]$ be a polynomial Poisson algebra
with an index set $T$.
If for each $i\in T$, there exists $\gamma(x_i)\in A$
such that $\ad_{x_i}^p=\ad_{\gamma(x_i)}$, then $A$ admits
a restricted Poisson structure $\pp{(-)}$
such that $\pp{x_i}=\gamma(x_i)$ for all $i\in T$.
\end{theorem}

\begin{proof} First we show that $A$ has a weakly
restricted Poisson structure, and then
verify that the weakly restricted Poisson structure
satisfies \eqref{E3.5.1}.

For the sake of simplicity, we assume
that $T=\{1,2,\cdots,n\}$.
To apply Lemma \ref{xxlem1.3}, we choose a canonical
monomial $\K$-basis of $A$, which is
$$\{x_1^{i_1}x_2^{i_2}\cdots x_n^{i_n}\mid i_1, i_2, \cdots,  i_n \ge 0\}.$$
We define $\pp{(x_1^{i_1}x_2^{i_2}\cdots x_n^{i_n})}$ inductively on
the degree $i_1+i_2+\cdots+i_n$ such that
$$\ad_{(x_1^{i_1}x_2^{i_2}\cdots x_n^{i_n})}^p=
\ad_{\pp{(x_1^{i_1}x_2^{i_2}\cdots x_n^{i_n})}},$$
and therefore get the restricted
Lie structure on $(A, \{-, -\})$ by Lemma \ref{xxlem1.3}. For
convenience, we denote $x^I=x_1^{i_1}x_2^{i_2}\cdots x_n^{i_n}$ and
$|I|=i_1+\cdots+i_n$ for $I=(i_1, \cdots, i_n)$.

If $|I|=0$, then $x^I=1$, we define $\pp{1}=0$ and if $|I|=1$,
then $x^I=x_i$ for some $1\le i\le n$. We define
$\pp{x_i}=\gamma(x_i)$ for each $1\le i\le n$. By hypothesis,
$\ad_{x^I}^p=\ad_{\pp{(x^I)}}$ for any $I$ with $|I|=0, 1$.

Proceeding by induction and assuming that $\pp{(x^I)}$ has been
defined such that
$\ad_{x^I}^p=\ad_{\pp{(x^I)}}$ for any $x^I$ with $|I|\le m$.
For each monomial $x^I$ of degree $m+1$, we assume that $k$ is
the smallest subscript such that $i_k\ge 1$ in $I$, i.e.
$I=(0, \cdots, 0, i_k, \cdots, i_{n})$ and define
\begin{align}
\label{E6.1.1}\tag{E6.1.1}
\pp{(x^I)}=& x_k^p\pp{(x_k^{i_k-1}x_{k+1}^{i_{k+1}}\cdots x_n^{i_n})}+
(x_k^{i_k-1}x_{k+1}^{i_{k+1}}\cdots x_n^{i_n})^p \pp{x_k}\\
\notag
&+\Phi_p(x_k, x_k^{i_1-1}x_{k+1}^{i_{k+1}}\cdots x_n^{i_n}).
\end{align}
By Lemma \ref{xxlem3.2}(1) for
$(x,y)=(x_k, x_k^{i_k-1}x_{k+1}^{i_{k+1}}\cdots x_n^{i_n})$
and the above definition, we have $\ad_{x^I}^p=\ad_{\pp{(x^I)}}$
for any $|I|=m+1$, which completes the induction. By
Lemma \ref{xxlem1.3}, $A$ has a weakly restricted Poisson structure.

Now let ${\bf b}$ be the set of all monomials, which is a
$\K$-basis of $A$. We prove that \eqref{E3.5.1} holds for
any pair of elements $(x, y)$ in ${\bf b}$ by induction on
$\deg x + \deg y$. If $x$ or $y$ is 1, then
\eqref{E3.5.1} holds trivially, which also takes care of
the case when $m := \deg x + \deg y\leq 1$.
Suppose that the assertion holds for $m$ and now assume
that $\deg x + \deg y = m + 1$. Let $xy = x_k^{i_k}
x_{k+1}^{i_{k+1}}\cdots x_{n}^{i_n}$ where $i_k > 0$.
By \eqref{E6.1.1}, the pair $(x_{k}, x_k^{i_k-1}
x_{k+1}^{i_{k+1}}\cdots x_{n}^{i_n})$ satisfies \eqref{E3.5.1}.
By symmetry, we may assume that $x=x_k g$. Then
the above says that the pair $(x_k, gy)$ satisfies
\eqref{E3.5.1}. By induction hypothesis, the pairs
$(x_k, g)$ and $(g,y)$ satisfy \eqref{E3.5.1}.
By Proposition \ref{xxpro4.6}(2), $(x,y)=(x_k g, y)$
satisfies \eqref{E3.5.1}. By induction, \eqref{E3.5.1}
holds for any two elements in ${\bf b}$. Finally the main
statement follows from Theorem \ref{xxthm4.7}.
\end{proof}

As a consequence, we have the following.

\begin{example}
\label{xxex6.2} Let $L$ be a restricted Lie algebra. We claim that
the polynomial Poisson algebra $A:=\K[L]$ (also denoted by $S(L)$)
is a restricted Poisson algebra. Let $\{x_i\}_{i\in I}$ be a basis
of $L$. Then, for each $i$, there is an $\gamma(x_i):=
x_i^{[p]}\in L$ such  that $\ad_{x_i}^p =\ad_{\gamma(x_i)}$
when restricted to $L$. Since $A$ is a polynomial ring
over $L$, both $\ad_{x_i}^p$ and $\ad_{\gamma(x_i)}$ extends
uniquely to derivations of $A$. Thus $\ad_{x_i}^p =
\ad_{\gamma(x_i)}$ holds when applying to $A$. The claim follows
from Theorem \ref{xxthm6.1} and there is a unique
restricted structure $\pp{(-)}$ on $A$ such that
$$\pp{x}=x ^{[p]}, \forall \; x\in L.$$
\end{example}

Let $V$ be a vector space. Then the free restricted Lie algebra $RLie(V)$
can be defined by using the universal property or by taking the
restricted Lie subalgebra of the free associative algebra generated
by $V$ with the $p$-map being the $p$-powering map. Now we can define
the free restricted Poisson algebra generated by $V$.

\begin{definition}
\label{xxdef6.3} Let $V$ be a $\K$-space. The free restricted Poisson
algebra generated by $V$ is defined to be
$$FRP(V)=\K[RLie(V)].$$
\end{definition}

The following universal property is standard
\cite[Lemma 1, p. 312]{Sh}.

\begin{lemma}
\label{xxlem6.4}
Let $A$ be a restricted Poisson algebra and $V$ be a vector space. Every
$\K$-linear map $g: V\to A$ extends uniquely to a restricted Poisson algebra
morphism $G: FRP(V)\to A$ such that $g$ factors through
$G$.
\end{lemma}

Continuing Example \ref{xxex6.2}, when $L$ is a restricted Lie
algebra over $\K$ and $S(L):=\K[L]$ the symmetric algebra on $L$,
then $S(L)$ admits an induced restricted Poisson structure.
One natural setting in positive characteristic  is to replace
the symmetric algebra $S(L)$ by the truncated (or small) symmetric
algebra $s(L)$. By definition, when $L$ has a $\K$-basis $\{x_i\}_{i\in I}$,
\begin{equation}
\label{E6.4.1}\tag{E6.4.1}
s(L)=\K[x_i\mid i\in I]/(x_i^p, \; \forall\;  i\in I).
\end{equation}
It is easily seen that $s(L)$ admits a Poisson structure with the bracket
$\{f, g\}=\sum\limits_{i, j}
(\pian{f}{x_i}\pian{g}{x_j}-\pian{f}{x_j}\pian{g}{x_i})\{x_i, x_j\}$
for any $f, g\in s(L)$. Next we show that $s(L)$ has a natural
restricted Poisson structure.

\begin{theorem}
\label{xxthm6.5}
Let $L$ be a restricted Lie algebra over $\K$ of
characteristic $p$ and let $s(L)$ be the Poisson algebra with the
bracket induced by $L$. Then $s(L)$ admits a natural restricted
Poisson structure induced by the $p$-map of $L$.
\end{theorem}

\begin{proof} By Example \ref{xxex6.2}, $S(L)$ has an induced
restricted Poisson algebra structure. By \eqref{E6.4.1},
$$s(L)=S(L)/J$$
where $J$ is the Poisson ideal generated by $x_i^p$ for all
$i\in I$. By Proposition \ref{xxpro3.5}(3), $\pp{(x_i^p)}=0$.
By Lemma \ref{xxlem4.2}, $J$ is a restricted Poisson ideal
as desired.
\end{proof}

\section{Restricted Poisson algebras from deformation quantization}
\label{xxsec7}

We will produce more examples in this section.

Let $A$ be a commutative (associative) algebra. Let $\K[[t]]$ be
the formal power series ring in one variable $t$.
A {\it formal deformation} of $A$ means an associative algebra
$A[[t]]$ over $\K[[t]]$ with multiplication, denoted by $m_t$,
satisfying
\[m_t(a\otimes b)=a\ast b=ab+m_1(a, b)t+\cdots+m_n(a, b)t^n+\cdots,\]
for all $a, b\in A\subset A[[t]]$. We should view $A[[t]]$ as
the power series ring in one variable $t$ with coefficients
in $A$ where the associative multiplication $m_t$ (or the star product
$\ast$) is induced by a family of $\K$-bilinear maps
$\{m_i\colon A\otimes A\to A\}_{i\ge 0}$ with $m_0(a, b)=ab$.

Define a bilinear map $\{-, -\}\colon A\otimes A\to A$ by
setting $\{a, b\}=m_1(a, b)-m_1(b, a)$. It is
easy to check that $A$ together with the bracket $\{-, -\}$
is a Poisson algebra. Then $(A, \{-,-\})$ is called the
\emph{classical limit} of $(A[[t]],m_t)$, and $(A[[t]], m_t)$
is called a \emph{deformation quantization} of the Poisson
algebra $(A, \{-,-\})$.

For every $f\in A$, we write
the $p$-power of $f$  as
\begin{equation}
\label{E7.0.1}\tag{E7.0.1}
f^{\ast p}=\Sum_{n=0}^\infty M_n^p(f) t^n
=f^p+M_1^p(f) t+M_2^p(f)t^2 +\cdots \quad \in A[[t]]
\end{equation}
where $M_i^p(f)\in A$ for all $i=0,1,2,\cdots$.

\begin{proposition}
\label{xxpro7.1}
Let $(A, \cdot, \{-,-\})$ be a Poisson algebra over $\K$
and let $(A[[t]], \ast)$ be a deformation
quantization of $A$. If $M_n^p(f)=0$ for $1\le n\le p-2$
and $f^p$ is central in $A[[t]]$ for all $f\in A$,
then $A$ admits a restricted Poisson structure.
\end{proposition}

\begin{proof} Recall that
$f\ast g=\Sum_{n=0}^\infty m_n(f, g)t^n\in A[[t]]$
for all $f, g\in A$, where $m_n(f,g)\in A$ for all
$n$. By the definition of the deformation quantization,
\[\{f, g\}=\lim_{t\to 0} \dfrac{f\ast g-g\ast f}{t}=m_1(f, g)-m_1(g, f).\]
for all $f, g\in A$.
Considering the Frobenius map $f\mapsto f^{\ast p}$ in $A[[t]]$, we get
\begin{equation}
\label{E7.1.1}\tag{E7.1.1}
[f^{\ast p}, g]_\ast=[\underbrace{f,\cdots, f}_{p\ {\rm copies}}, g]_\ast
\end{equation}
for all $f, g\in A$.

Since $[f, g]_\ast=\{f, g\}t \  ({\rm mod}\ t^2)$ and
$[-, -]_\ast$ is $\K[[t]]$-bilinear, we have
\[ [\underbrace{f,\cdots, f}_{p\ {\rm copies}}, g]_\ast
\equiv \{\underbrace{f,\cdots, f}_{p\ {\rm copies}}, g\}t^{p}
\  ({\rm mod}\  t^{p+1}).\]
By assumption, $M_n^{p}(f)=0$ for $1\le n\le p-2$ and $f^p$
is central in $A[[t]]$.
Using the fact that \eqref{E7.1.1} or $\ad_{f^{*p}}(g)=(\ad_{f})^p(g)$,
it follows that
$$\{ M_{p-1}^p(f), g\} t^p
=\{\underbrace{f,\cdots, f}_{p\ {\rm copies}}, g\}
 t^p \  ({\rm mod}\ t^{p+1})$$
or
\begin{equation}
\label{E7.1.2}\tag{E7.1.2}
\{M^p_{p-1}(f), g\}=m_1(M^p_{p-1}(f), g)-m_1(g, M^p_{p-1}(f))
=\{\underbrace{f,\cdots, f}_{p\ {\rm copies}}, g\}
\end{equation}
for all $g\in A$. We define $\pp{f}=M_{p-1}^p(f)$ for any $f\in A$,
and prove that the map $f\mapsto M_{p-1}^p(f)$ gives rise to a
restricted Poisson structure on $A$.

Note that Definition \ref{xxdef1.1}(1) follows from \eqref{E7.1.2}.
Definition \ref{xxdef1.1}(2) follows from the fact that $(\lambda f)^{*p}
=\lambda^p f^{*p}$. For condition in Definition \ref{xxdef1.1}(3),
we consider the Frobenius map of $A[[t]]$, and get a restricted
Lie structure of $(A[[t]], [-,-]_\ast)$. It follows from Example
\ref{xxex1.2} that
\[(f+g)^{\ast p}-f^{\ast p}-g^{\ast p}=\Lambda^\ast_p(f, g).\]
Computing the coefficients of $t^{p-1}$ of the above equation,
we get
\begin{align}
\label{E7.1.3}\tag{E7.1.3}
\pp{(f+g)}-\pp{f}-\pp{g}=\Lambda_p(f, g)
\end{align}
as desired.

Finally it remains to show \eqref{E3.4.1}.
By assumption, $M_n^p(f)=0$ for all $1\le n\le p-2$.
We compute the coefficient of $t^{p-1}$ in the
expression of $f^{\ast, 2p}$ as follows:
\begin{align*}
f^{\ast, 2p} &= f^{\ast p}\ast f^{\ast p}\\
&= (f^p+t^{p-1} M_{p-1}^p(f)+\cdots)\ast (f^p+t^{p-1}M_{p-1}^p(f)+\cdots)\\
&\equiv f^{2p}+2f^p M_{p-1}^p(f) t^{p-1}\ \ ({\rm mod}\  t^{p})
\end{align*}
Assume that $f\ast f=f^2+tW$, where $W=m_1(f, f)+m_2(f, f)t+\cdots$,
and it follows that
\begin{align*}
f^{\ast, 2p} &= (f^{\ast 2})^{\ast p}= (f^2+tW)^{\ast p}\\
&=(f^2)^{\ast p}+ (tW)^{\ast p}+\Lambda_p^\ast(f^2, tW)\\
&\equiv f^{2p}+M_{p-1}^p(f^2)t^{p-1}\ \ ({\rm mod}\  t^{p})
\end{align*}
Therefore, for all $f\in A$, $\pp{f^2}=2f^p\pp{f}$, which is
\eqref{E3.4.1}.
\end{proof}

We now give some explicit examples.

\begin{example}\label{xxex7.2}
Let $A=\K[x, y]$ be a Poisson algebra over a field $\K$ of
characteristic $p\geq 3$ with the bracket
given by $\{x, y\}=1$.

Let $\mu$ be the multiplication of the commutative algebra
$A[[t]]$. By a direct calculation, the Poisson algebra $A$ admits a
deformation quantization $(A[[t]], \ast)$ with the star product given by
\[f\ast g=\mu(\exp(t(\partial_1\ot \partial_2))(f\ot g))\]
for all $f, g\in A$, where $\partial_1$ and $\partial_2$
are the partial derivatives of $f$ with respect to the variables
$x$ and $y$, respectively. To be precise, we have
\[f\ast g=\Sum_{0\le n\le p-1}m_n(f, g)t^n=
\Sum_{0\le n\le p-1}\dfrac{t^n}{n!} (\partial^n_1 f) (\partial^n_2 g).\]
Clearly, $f^p\ast g=f^pg= g\ast f^p$ for any $f, g\in A$ and hence $f^p$
is central in $A[[t]]$. Moreover,
for every $f\in A$, we claim that
\begin{equation}
\label{E7.2.1}\tag{E7.2.1}
{\text{$M_n(f)=0$ for $1\le n\le p-2$.}}
\end{equation}
The proof of the above is given in Appendix.
By Proposition \ref{xxpro7.1}, $A$ admits a restricted Poisson structure
with the $p$-map $\pp{f}=M_{p-1}^p(f)$ for any $f\in A$.
The $p$-map agrees with \eqref{E4.8.2} when $p=3$
and \eqref{E4.8.3} when $p=5$.
\end{example}

The next is a generalization of the previous example.

\begin{example}
\label{xxex7.3}
Let $A=\K[x_1, \cdots, x_m]$ be a Poisson algebra with
the bracket given by $\{x_i, x_j\}=c_{ij}\in \K$
for $1\le i< j\le n$. By direct calculation,
a deformation quantization $(A[[t]], \ast)$ of the
Poisson algebra $A$ is given by
\[f\ast g=\mu
\left(\exp\left(t\Sum_{1\le i<j\le m} c_{ij}
\partial_i \ot \partial_j\right)(f\ot g)\right)\]
for all $f, g\in A$, where $\partial_i$ is the partial
derivative of $f$ with respect to the variable $x_i$.
This is well-defined by Remark \ref{xxrem10.1}(2).
Clearly, $f^p\in A\subset A[[t]]$ is central for any
$f\in A$. Being similar to the proof of Example \ref{xxex7.2}
in Appendix, we have $M_n(f)=0$ for $1\le n\le p-2$
and all $f\in A$. By Proposition \ref{xxpro7.1}, $A$
admits a restricted Poisson structure with the $p$-map
$\pp{f}=M_{p-1}^p(f)$ for any $f\in A$. When $p=3$, the
$p$-map is given in Example \ref{xxex4.10}.
\end{example}

\begin{example}
\label{xxex7.4}
Let $B_{2n}=\K[x_1, \cdots, x_{2n}]/I$ be the 
$p$-truncated polynomial Poisson algebra in $2n$
variables over $\K$, where the Poisson
bracket is defined by
$$\{f, g\}=\Sum_{i=1}^n (\partial_i(f)\partial_{n+i}(g)
-\partial_{n+i}(f)\partial_i(g))$$
for all $f, g\in B_{2n}$,
and $I$ is generated by $x_i^p, i=1, \cdots, 2n$.
In \cite{Sk}, Skryabin introduced the notion of
the normalized $p$-map on $(B_{2n}, \{-,-\})$,
say, $\pp{1}=0$ and $\pp{f}\in \mathfrak{m}^2$
for all $f\in \mathfrak{m}^2$, where $\mathfrak{m}$
is the maximal ideal of $B_{2n}$ as an associative algebra.

We consider the Poisson algebra
$A=\K[x_1, \cdots, x_{2n}]$ in Example \ref{xxex7.3}
with the bracket given by $c_{ij}=\delta_{i+n, j}$
for all $1\le i<j\le 2n$. Clearly, $x_i^p$ is
central and $I$ is a Poisson ideal of $A$. By Proposition
\ref{xxpro3.5}(3), $\pp{(x_i^p)}=0$ for all $i\in I$,
and by Lemma \ref{xxlem4.2}, $I$ is a restricted
Poisson ideal of $A$. Therefore, it follows from
Proposition \ref{xxpro4.3} that the Poisson
algebra $B_{2n}$ admits a restricted Poisson
structure. Clearly, this $p$-map is normalized.
\end{example}


\section{Connection with restricted Lie-Rinehart Algebras}
\label{xxsec8}
Some definitions concerning Lie-Rinehart algebras were given
in Section 2. Let $A$ be a Poisson algebra and $\Omega_{A/\K}$ its
K\" ahler differentials. Then the pair $(A, \Omega_{A/\K})$ is a
Lie-Rinehart algebra over $\K$, where the anchor map
$\alpha: \Omega_{A/\K} \to \Der(A)$ is given in \eqref{E2.2.2}.
Dokas introduced the notion of a restricted Lie-Rinehart
algebra and study its cohomology theory in \cite{Do}. The goal
of this section is to show that the Lie-Rinehart algebra
$(A, \Omega_{A/\K})$ admits a natural restricted structure
if the Poisson algebra $A$ is weakly restricted and $\Omega_{A/\K}$
is a free module over $A$.

Let $(L, (-)^{[p]}$ and $(L', (-)^{[p]})$ be restricted Lie
algebras. A map $f\colon (L, (-)^{[p]}) \to (L', (-)^{[p]})$
is called a restricted Lie homomorphism, if $f$ is a Lie
algebra homomorphism and satisfies $f(x^{[p]})=f(x)^{[p]}$
for all $x\in L$.

The following definition was introduced by Dokas \cite{Do}.

\begin{definition}
\cite[Definition 1.7]{Do}
\label{xxdef8.1}
A \emph{restricted Lie-Rinehart algebra} $(A, L, (-)^{[p]})$ over
a commutative $\K$-algebra $A$, is a Lie-Rinehart algebra over $A$
such that
\begin{enumerate}
\item[(a)]
$(L, (-)^{[p]})$ is a restricted Lie algebra over $\K$,
\item[(b)]
the anchor map
is a restricted Lie homomorphism, and
\item[(c)]
the following relation holds:
\[(aX)^{[p]}=a^pX^{[p]}+(aX)^{p-1}(a)X\]
for all $a\in A$ and $X\in L$.
\end{enumerate}
\end{definition}

We now prove Theorem \ref{xxthm0.5}.

\begin{theorem}
\label{xxthm8.2}
Let $(A, \cdot, \{-,-\}, \pp{(-)})$ be a weakly restricted Poisson
algebra. If the K\"ahler differential $\Omega_{A/\K}$ is a free,
then the Lie-Rinehart algebra $(A, \Omega_{A/\K}, (-)^{[p]})$ is
restricted, where the $p$-map of $\Omega_{A/\K}$ is defined by
\begin{align*}
&(x\d u)^{[p]}= x^p \d\pp{u}+(x \d u)^{p-1}(x)\d u,
\end{align*}
for all $x\d u \in \Omega_{A/\K}$.
\end{theorem}

\begin{proof} Since $\Omega_{A/\K}$ is a free $A$-module,
$\Omega_{A/\K}$ can be embedded into the universal enveloping
algebra $\U(A, \Omega_{A/\K})$ [Lemma \ref{xxlem2.3}]. By the
proof of \cite[Proposition 2.2]{Do}, it suffices to show that
\[\ad_{x\d u}^p(y\d v)=[x^p \d\pp{u}+(x \d u)^{p-1}(x)\d u, y\d v]\]
for all $x\d u, y\d v\in \Omega_{A/\K}$.

By Hochschild's relation in \cite[Lemma 1]{Ho1}, we get in $\U(A, L)$
the relation
\[(\iota_2(x\d u))^p=\iota_1(x^p)(\iota_2(\d u))^p+\iota_2((x\d u)^{p-1}(x)\d u)\]
for all $x\d u\in \Omega_{A/\K}$.
Considering the Frobenius map of $\U(A, L)$, we have
\begin{align*}
[(\iota_2(\d u))^p, \iota_1(y)]
=  [\iota_2(\d u), \cdots, \iota_2(\d u), \iota_1(y)]
=  \iota_1( (\ad_u)^p(y)),
\end{align*}
and hence $\iota_2(\d u)^p \iota_1(y)=\iota_1(y)\iota_2(\d u)^p+\iota_1( (\ad_u)^p(y))$
for all $\d u\in \Omega_{A/\K}, y\in A$.
Moreover, for $x\d u, y\d v\in \Omega_{A/\K}\subset \U(A, L)$,
\begin{align*}
[\iota_1(x^p)(\iota_2(\d u))^p, \iota_2(y\d v)]
=& \iota_1(x^p)(\iota_2(\d u))^p
 \iota_1(y)\iota_2(\d v)-\iota_1(y)\iota_2(\d v)\iota_1(x^p)(\iota_2(\d u))^p\\
=& \iota_1(x^p)(\iota_1(y)(\iota_2(\d u))^p+\iota_1(\ad_u)^p(y))\iota_2(\d v)\\
&\qquad -\iota_1(y)(\iota_1(x^p)\iota_2(\d v)+\iota_1(\{v, x^p\})(\iota_2(\d u))^p\\
=& \iota_1(x^py)[(\iota_2(\d u))^p, \iota_2(\d v)]+\iota_2(x^p(\ad_u)^p(y)\d v)\\
=& \iota_1(x^py)\iota_2(\ad_{\d u}^p(\d v))+\iota_2(x^p(\d u)^p(y)\d v)\\
=& \iota_1(x^py)\iota_2(\d (\ad_u^p(v)))+\iota_2(x^p(\ad_u)^p(y)\d v),
\end{align*}
and therefore,
\begin{align*}
\iota_2(\ad_{x\d u}^p(y\d v))
=& [(\iota_2(x\d u))^p, \iota_2(y\d v)]\\
=& [\iota_1(x^p)(\iota_2(\d u))^p+\iota_2((x\ad_u)^{p-1}(x)\d u), \iota_2(y\d v)]\\
=& \iota_1(x^py)\iota_2(\d (\ad_u^p(v)))+\iota_2(x^p(\ad_u)^p(y)\d v)\\
&\qquad\qquad\qquad +\iota_2([(x\ad_u)^{p-1}(x)\d u, y\d v])\\
=& \iota_2(x^py\d (\ad_u^p(v)))+\iota_2(x^p(\ad_u)^p(y)\d v)\\
&\qquad\qquad\qquad +\iota_2([(x\ad_u)^{p-1}(x)\d u, y\d v])\\
=& \iota_2([x^p\d\pp{u}+(x\ad_u)^{p-1}(x)\d u, y\d v]),
\end{align*}
and hence $\ad_{x\d u}^p(y\d v)=[x^p\d\pp{u}+(x\ad_u)^{p-1}(x)\d u, y\d v]$
as desired.
\end{proof}

For Poisson algebras $A$ in Examples \ref{xxex4.8}-\ref{xxex4.10},
\ref{xxex6.2}, Theorem \ref{xxthm6.5},
Examples \ref{xxex7.2}-\ref{xxex7.4}, it is automatic that
$\Omega_{A/\Bbbk}$ is free over $A$.

\section{Restricted Poisson Hopf algebras}
\label{xxsec9}

We first recall the definition of Poisson Hopf algebras. The notion
of a Poisson Hopf algebra was probably first introduced by Drinfel'd
\cite{Dr1, Dr2} in 1980s, see also \cite{DHL}.

\begin{definition}
\label{xxdef9.1}
Let $A$ be a Poisson algebra. We say that $A$ is a Poisson Hopf
algebra if
\begin{enumerate}
\item[(1)]
$A$ is a Hopf algebra with usual operations $\Delta, \epsilon,S$.
\item[(2)]
$\Delta: A\to A\otimes A$ and $\epsilon: A\to \Bbbk$ are Poisson
algebra morphisms and $S: A\to A$ is a Poisson algebra anti-automorphism.
\end{enumerate}
\end{definition}

To define restricted Poisson Hopf algebras, we need first
show that tensor product of two restricted Poisson algebras is again
a restricted Poisson algebra.

\begin{proposition}
\label{xxpro9.2} Let $A$ and $B$ be two restricted Poisson algebras.
Then there is a unique restricted Poisson structure on $A\otimes B$
such that
\begin{equation}
\label{E9.2.1}\tag{E9.2.1}
\pp{(a\otimes b)}=\pp{a}\otimes b^p+a^p\otimes \pp{b}
\end{equation}
for all $a\in A$ and $b\in B$.
\end{proposition}

\begin{proof} First of all, it is well-known that $A\otimes B$ is a
Poisson algebra with bracket defined by
$$\{a_1\otimes b_1, a_2\otimes b_2\}=\{a_1,a_2\}\otimes b_1b_2
+a_1a_2\otimes \{b_1, b_2\}$$
for all $a_1,a_2\in A$ and $b_1,b_2\in B$.

Let $\{a_i\}_{i\in I}$ (respectively, $\{b_j\}_{j\in J}$)
be a $\Bbbk$-basis of $A$ (respectively, $B$) and assume that
$1_A\in \{a_i\}_{i\in I}$ and $1_B\in \{b_j\}_{i\in J}$.
Then $\{a_i \otimes b_j\}_{i\in I,j\in J}$
is a $\Bbbk$-basis of $A\otimes B$.

For any $a\in A$ and $b\in B$, $\ad_{a\otimes b}^p$ is a
derivation. For any $c\otimes d\in A\otimes B$, we have
$$\begin{aligned}
\ad_{a\otimes b}^p (c\otimes d)&= (1\otimes d)\ad_{a\otimes b}^p(c\otimes 1)
+(c\otimes 1)\ad_{a\otimes b}^p(1\otimes d)\\
&=(1\otimes d)(\ad_{a}^p(c)\otimes b^p)
+(c\otimes 1)(a^p\otimes \ad_{b}^p(d))\\
&=(1\otimes d)(\ad_{\pp{a}}(c)\otimes b^p)
+(c\otimes 1)(a^p\otimes \ad_{\pp{b}}(d))\\
&=(1\otimes d)(\ad_{\pp{a}\otimes b^p}(c\otimes 1))
+(c\otimes 1)(\ad_{a^p\otimes \pp{b}}(1\otimes d))\\
&=(1\otimes d)(\ad_{\pp{a}\otimes b^p}(c\otimes 1))
+(c\otimes 1)(\ad_{\pp{a}\otimes b^p}(1\otimes d))\\
& \qquad +(1\otimes d)(\ad_{a^p\otimes \pp{b}}(c\otimes 1))
+(c\otimes 1)(\ad_{a^p\otimes \pp{b}}(1\otimes d))\\
&=\ad_{\pp{a}\otimes b^p}(c\otimes d)
+\ad_{a^p\otimes \pp{b}}(c\otimes d)\\
&=\ad_{\pp{a}\otimes b^p+a^p\otimes \pp{b}}(c\otimes d).
\end{aligned}
$$
In particular,
$$\ad_{a_i\otimes b_j}^p =\ad_{(\pp{a_i}\otimes b_j^p+a_i^p\otimes \pp{b_j})}$$
for all $i$ and $j$.
Since $\{a_i \otimes b_j\}_{i\in I,j\in J}$
is a $\Bbbk$-basis of $A\otimes B$, by Lemma \ref{xxlem1.3}, there is a unique
weak restricted Poisson structure on $A\otimes B$ such that
\begin{equation}
\label{E9.2.2}\tag{E9.2.2}
\pp{(a_i\otimes b_j)}=\pp{a_i}\otimes b_j^p+a_i^p\otimes \pp{b_j}
\end{equation}
for all $i,j$, which agrees with \eqref{E9.2.1}. It remains to show that
this weak restricted Poisson structure on $A\otimes B$ is indeed a
restricted Poisson structure and \eqref{E9.2.1} holds.

We first prove \eqref{E9.2.1}. By \eqref{E9.2.2}, $\pp{(a_i\otimes 1)}=
\pp{a_i}\otimes 1$. It follows from Definition \ref{xxdef1.1} that
\begin{equation}
\label{E9.2.3}\tag{E9.2.3}
\pp{(a\otimes 1)}=\pp{a}\otimes 1
\end{equation}
for all $a\in A$. By symmetry,
$\pp{(1\otimes b)}=1\otimes \pp{b}$ for all $b\in B$.
Since $\{a_i\otimes 1, 1\otimes b_j\}=0$, \eqref{E9.2.2} implies that
the pair $(a_i\otimes 1, 1\otimes b_j)$ satisfies \eqref{E3.5.1}.
By Proposition \ref{xxpro4.6}(4), $R_{a_i\otimes 1}$ is a $\Bbbk$-vector
space; and by assumption, $\{b_j\}$ is a $\Bbbk$-basis of $B$, we have that
$R_{a_i\otimes 1}\supseteq B$. Or, for any $b\in B$, the pair
$(a_i\otimes 1, 1\otimes b)$ satisfies \eqref{E3.5.1}. By switching $a$ and
$b$ and applying the same argument, one sees that any pair
$(a\otimes 1, 1\otimes b)$ satisfies \eqref{E3.5.1}.
This means that
$$\begin{aligned}
\pp{(a\otimes b)}&=\pp{(a\otimes 1)} (1\otimes b)^p+
(a\otimes 1)^p \pp{(1\otimes b)}+\Phi_p(a\otimes 1, 1 \otimes b)\\
&= \pp{(a\otimes 1)} (1\otimes b)^p+
(a\otimes 1)^p \pp{(1\otimes b)}\\
&=\pp{a}\otimes b^p+a^p\otimes \pp{b}.
\end{aligned}
$$
So we proved \eqref{E9.2.1}.

For the rest, we claim that for any pair of elements
$(a_i\otimes b_j, a_k\otimes b_l)$, \eqref{E3.5.1} holds.
By using \eqref{E9.2.3}, \eqref{E3.5.1} holds for all
pairs of the form $(a\otimes 1, a'\otimes 1)$.
By symmetry, \eqref{E3.5.1} holds for all
pairs of the form $(1\otimes b, 1\otimes b')$. By
\eqref{E9.2.1}, \eqref{E3.5.1} holds for pairs of the
form $(a\otimes 1, 1 \otimes b)$.
Set $f=a\otimes 1, g=a'\otimes 1$ and $h=1\otimes b$
for any $a,a'\in A$ and $b\in B$.
Then $(f,g)$, $(g,h)$ and $(fg,h)$ satisfy \eqref{E3.5.1}.
By Proposition \ref{xxpro4.6}(2), $(f,gh)$ satisfies
\eqref{E3.5.1}. Or equivalently, $(a\otimes 1, a'\otimes b)$
satisfies \eqref{E3.5.1}. By symmetry,
$(1\otimes b, a\otimes b')$, $(a\otimes b, a'\otimes 1)$
and $(a\otimes b, 1\otimes b')$ satisfy \eqref{E3.5.1}.
Recycle the letters and let $f=a\otimes b$, $g=a'\otimes 1$ and
$h=1\otimes b'$. We have that $(f,g)$, $(g,h)$ and $(fg, h)$
all satisfy \eqref{E3.5.1}. By Proposition \ref{xxpro4.6}(2),
$(f,gh)$ satisfies \eqref{E3.5.1}. By choosing special
$a,a',b,b'$ we have that $(a_i\otimes b_j, a_k\otimes b_l)$
satisfies \eqref{E3.5.1} as desired. This says that every
pair of elements from the $\Bbbk$-basis $\{a_i\otimes b_j\}_{i\in I,j\in J}$
satisfies \eqref{E3.5.1}. By Theorem \ref{xxthm4.7}, the weak
restricted Poisson structure on $A\otimes B$ is actually a
restricted Poisson structure.

The above proof shows that there is a unique restricted Poisson
structure on $A\otimes B$ satisfying \eqref{E9.2.2}. Since
\eqref{E9.2.1} is a consequence of \eqref{E3.5.1}, the assertion
follows.
\end{proof}

Now it is reasonable to define a restricted Poisson Hopf
algebra.

\begin{definition}
\label{xxdef9.3} A restricted Poisson algebra $H$ is called a
{\it restricted Poisson Hopf algebra} if there are restricted
Poisson algebra maps $\Delta: H\to H\otimes H$, $\epsilon: H\to \K$
and restricted Poisson algebra anti-automorphism $S: H\to H$ such
that $H$ together with $(\Delta, \epsilon, S)$ becomes a Hopf
algebra.
\end{definition}

One canonical example is the following.

\begin{example}
\label{xxex9.4}
Let $L$ be a restricted Lie algebra. Then $s(L)$
(given in Theorem \ref{xxthm6.5}) is
a restricted Poisson Hopf algebra with the structure
maps determined by
$$\begin{aligned}
\Delta: \quad & x\to x\otimes 1+1\otimes x,\\
\epsilon:\quad & x\to 0,\\
S: \quad & x\to -x
\end{aligned}
$$
for all $x\in L$. It is straightforward to check that $s(L)$ is a
restricted Poisson Hopf algebra. Similarly,
$S(L)$ (given in Example \ref{xxex6.2}) is a
a restricted Poisson Hopf algebra with structure maps
determined as above.
\end{example}

\section{Appendix: The proof of \eqref{E7.2.1}.}
\label{xxsec10}


Let $A=\K[x, y]$ be the Poisson algebra with the Poisson bracket
determined  by
$$\{x, y\}=1.$$
Recall from Example \ref{xxex7.2}
that the deformation quantization of $A$ is isomorphic to
an associative algebra
$(A[[t]], \ast)$ such that the star product of $f, g\in A\subset A[[t]]$
is given by
\begin{equation}
\label{E10.0.1}\tag{E10.0.1}
f\ast g=\mu(\exp(t(\pian{}{x}\ot \pian{}{y}))(f\ot g))
=\mu\left( \sum_{i=0}^{\infty} \frac{t^i}{i!} \; \frac{\partial^i f}{\partial x^i}
\otimes \frac{\partial^i g}{\partial y^i}\right),
\end{equation}
where $\mu\colon A\ot A\to A$ is the multiplication operation of $A$.
Define a sequence of Hasse-Schmidt derivations (or divided power derivations)
$$\partial^{(i)}_1=\frac{1}{i!} \left( \frac{\partial}{\partial x}\right)^i
\quad {\text{and}}\quad
\partial^{(i)}_2=\frac{1}{i!} \left( \frac{\partial}{\partial y}\right)^i, \quad
\forall \; i\geq 0.$$
Then all of them are $\Bbbk$-linear operations from $A$ to $A$. Using these
we can re-write part of \eqref{E10.0.1} as
\begin{equation}
\label{E10.0.2}\tag{E10.0.2}
\exp(t(\pian{}{x}\ot \pian{}{y}))(f\ot g)
=\sum_{i=0}^{\infty} i! \partial^{(i)}_1(f) \partial^{(i)}_2(g)
= \sum_{i=0}^{p-1} i! \partial^{(i)}_1(f) \partial^{(i)}_2(g),
\end{equation}
which is a sum of finitely many terms. Therefore \eqref{E10.0.1}
is well-defined and the summation in \eqref{E10.0.1} is finite.

\begin{remark}
\label{xxrem10.1}
Consider a generalization of \eqref{E10.0.1} in $n$ variables.
Let $B=\Bbbk[ x_1,\cdots,x_n]$ and $\partial_i=\frac{\partial}{\partial x_i}$
for $i=1,\cdots,n$.
\begin{enumerate}
\item[(1)]
For each $c_{ij}\in \Bbbk$, $\exp( t\; c_{ij} \; \partial_i\otimes \partial_j)
(f\otimes g)$ is well-defined for all $f,g\in B$, and it is a sum of finitely
many terms as in \eqref{E10.0.2}.
\item[(2)]
For a set of $\{c_{ij}\}_{1\leq i,j\leq n}$,
\begin{equation}
\label{E10.1.1}\tag{E10.1.1}
\exp(t\sum_{i,j} c_{ij} \partial_i\otimes \partial_j)
(f\otimes g)
=\prod_{i,j}
\left(\exp(t c_{ij} \partial_i\otimes \partial_j) \right) (f\otimes g),
\end{equation}
which is well-defined for all $f,g\in B$ and is a sum of finitely
many terms in a similar fashion as \eqref{E10.0.2} (but more than $p$ terms
in general).
\end{enumerate}
\end{remark}

We now go back to the case of two variables. Clearly,
$$\partial \circ \mu=\mu(\partial \otimes \id +\id\otimes \partial)$$
for $\partial=\pian{}{x}$ or $\pian{}{y}$,
and hence
\[(\pian{}{x}\ot \pian{}{y})^n(\mu\ot \id)=
(\mu\ot \id)(\pian{}{x}\ot \id\ot \pian{}{y}
+\id\ot \pian{}{x}\ot \pian{}{y})^n\]
for all $n\ge 1$. It follows that
\[\exp(t\pian{}{x}\ot \pian{}{y})(\mu\ot \id)
=(\mu\ot \id)\exp(t(\pian{}{x}\ot \id \ot \pian{}{y}
+\id\ot \pian{}{x}\ot \pian{}{y})),\]
and therefore,
\begin{align*}
&(f\ast g)\ast h
= \mu(\exp(t\pian{}{x}\ot \pian{}{y})((f\ast g)\ot h)\\
&= \mu(\exp(t\pian{}{x}\ot \pian{}{y})
((\mu \ot \id)\exp(t\pian{}{x}\ot \pian{}{y}\ot\id)(f\ot g\ot h)))\\
&=\mu(\mu\ot \id)(\exp(t(\pian{}{x}\ot \id \ot \pian{}{y}
+\id \ot \pian{}{x}\ot\pian{}{y}+\pian{}{x}\ot \pian{}{y}\ot \id))
(f\ot g\ot h))
\end{align*}
In general, for $k\geq 2$ and for $f_1, \cdots, f_k\in A\subset A[[t]]$,
\begin{align}\label{E10.1.2}\tag{E10.1.2}
f_1\ast \cdots\ast f_k=\mu^k(\exp(t\sum_{1\le i<j\le k}
\partial^i_j(f_1\ot \cdots\ot f_k))).
\end{align}
where $\partial^i_j=\id^{\ot i-1}\ot \pian{}{x}
\ot \id^{\ot j-i-1} \ot \pian{}{y} \ot \id^{\ot k-j}$
is a map from $A^{\ot k}$ to itself for all $1\le i<j\le k$,
$\mu^2(a\ot b)=(ab)$ for all $a, b\in A$ (extended to a
commutative multiplication on $A[[t]]$), and
$\mu^k=\mu^2(\mu^{k-1}\ot \id)$, $k\ge 3$.

Denote by $M_n^p(f)$ the coefficient of $t^n$ in $f^{\ast p}\in A[[t]]$,
see \eqref{E7.0.1}. For simplicity, we denote the map
\[\Phi^{i_1,\cdots,i_n}_{j_1,\cdots,j_n}=
\mu^p\circ(\partial^{i_1}_{j_1} \circ \cdots
\circ \partial^{i_n}_{j_n})\colon A^{\ot p}\to A\]
for $1\le i_t<j_t\le p, t=1,\cdots, n, n\ge 1$.
It follows from the equation \eqref{E10.1.2} that
\begin{equation}
\label{E10.1.3}\tag{E10.1.3}
M_n^p(f)=\dfrac{1}{n!}\sum_{1\le i_r<j_r\le p\atop r=1,\cdots, n}
\Phi^{i_1,\cdots,i_n}_{j_1,\cdots,j_n}(f^{\ot p})
\end{equation}
for all $0\leq n \leq p-1$.

\begin{claim}\label{xxcl10.2}
Retain the above notation. Then $M_n^p(f)=0$ for all $1\le n\le p-2$
and all $f\in A$.
\end{claim}

This appendix is devoted to the proof of Claim \ref{xxcl10.2}.
We need more notations.

Recall that an oriented graph  is a pair $G=(V(G), E(G))$,
where $V(G)$ is the set of vertices and $E(G)$ is the set
of edges. For $\al\in E(G)$, we denote by $s(\al)$ and
$t(\al)$ the source and the target of $\al$, respectively.

\begin{definition}
\label{xxdef10.3}
An  oriented graph $G=(V(G), E(G))$ is called a
\emph{totally ordered graph} (called \emph{tograph} for short),
if the set $V(G)$ of vertices is totally ordered set with the
ordering $\le$ and for every $\al\in E(G)$, $s(\al)<t(\al)$.
\end{definition}

Let $G$ be a tograph (possibly with multiple edges). Being
similar to usual oriented graphs, for each $v\in V(G)$, we
denote the {\it indegree} of $v$ by
$$d_G^+(v)=\#\{\al \in E(G)\mid t(\al)=v\},$$
the {\it outdegree} of $v$ by
$$d_G^-(v)=\#\{\al \in E(G)\mid s(\al)=v\},$$
the {\it degree} of $v$ by
$$d_G(v)=d_G^+(v)+d_G^-(v).$$
For $u, v\in V(G)$, we denote by $\nu(u, v)$ the number of
the edges with the source $u$ and the target $v$, i.e.
$\nu_G(u, v)=\#\{\al\in E(G)\mid s(\al)=u, t(\al)=v\}$.

{Let $G$ and $G'$ be tographs. A bijection $f\colon V(G)\to V(G')$ is called an isomorphism,
if $f$ preserves the order of vertices and $\nu_G(u, v)=\nu_{G'}(f(u), f(v))$ for all $u, v\in V(G)$.
Two tographs $G$ and $G'$ are
said to be \emph{isomorphic}, denoted by $G\cong G'$, provided that
there exists an isomorphism between $G$ and $G'$ .
Clearly, the automorphism group of a tograph $G$
is a trivial group since $f$ preserves the order of vertices and $V(G)$ is totally ordered.}
Suppose that $G_1, \cdots, G_k$ and $G'_1, \cdots, G'_m$ are
the connected components of $G$ and $G'$, respectively. Two
tographs $G$ and $G'$ are said to be \emph{equivalent}, and
denoted by $G\sim G'$, if $m=k$ and there exists a permutation
$\sigma\in \mathbb{S}_m$ such that $G_i$ and $G'_{\sigma(i)}$
are isomorphic for each $i=1, \cdots, m$.

Denote
$$\Gamma_n=\{(i_1, \cdots, i_n; j_1, \cdots, j_n)
\mid 1\le i_t<j_t\le p, t=1, \cdots, n\}.$$
For each given $(i_1, \cdots, i_n; j_1, \cdots, j_n)\in \Gamma_n$,
we can assign a tograph, denoted by
$G({{i_1, \cdots, i_n}\atop{j_1, \cdots, j_n}})$,
where
\begin{enumerate}
\item[$\bullet$]
the set of vertices
$V(G({{i_1, \cdots, i_n}\atop{j_1, \cdots, j_n}}))=\{1, 2, \cdots, p\}$
with the usual ordering of natural numbers, and
\item[$\bullet$]
the set of edges
$E(G({{i_1, \cdots, i_n}\atop{j_1, \cdots, j_n}}))
=\{(i_t, j_t)\mid t=1, \cdots, n\}$.
\end{enumerate}
We denote by $\mathcal{G}_n$ the set of tographs
$G({{i_1, \cdots, i_n}\atop{j_1, \cdots, j_n}})$
for all $(i_1, \cdots, i_n; j_1, \cdots, j_n)\in \Gamma_n$.

We consider the lexicographical order on the
set $\{(i, j)\mid 1\le i<j\le p\}$. To be precise,
$(i, j)<(i', j')$ if and only if $i<i'$ or $i=i', j<j'$.

Let $G$ and $G'$ be the tographs associated to elements
$(i_1, \cdots, i_n;$ $j_1, \cdots, j_n)$ and
$(i'_1, \cdots, i'_n$; $j'_1,\cdots, j'_n) \in \Gamma_n$,
respectively. Clearly, $G=G'$ if and only if there exists
a permutation $\sigma\in \mathbb{S}_n$ such that
$(i'_k, j'_k)=(i_{\sigma(k)}, j_{\sigma(k)})$ for all $k=1,\cdots, n$.
Therefore, for each $(i_1, \cdots, i_n; j_1, \cdots, j_n) \in \Gamma_n$,
there exists a permutation $\sigma\in \mathbb{S}_n$ such that
$G({{i_1, \cdots, i_n}\atop{j_1, \cdots, j_n}})
=G({{i_{\sigma(1)}, \cdots, i_{\sigma(n)}}\atop{j_{\sigma(1)},
\cdots, j_{\sigma(n)}}})$ with $(i_{\sigma(1)}, j_{\sigma(1)})
\le \cdots \le (i_{\sigma(n)}, j_{\sigma(n)})$.

\begin{lemma}
\label{xxlem10.4}
Retain the above notation.
\begin{enumerate}
\item[(1)]
Let $G$ be the tograph associated to
$(i_1, \cdots, i_n; j_1, \cdots, j_n)\in \Gamma_n$. Then
\begin{align}
\label{E10.4.1}\tag{E10.4.1}
\Phi^{i_1,\cdots, i_n}_{j_1,\cdots, j_n}(f^{\ot p})=
\frac{\partial^{d_G(1)}f}{\partial x^{d_G^-(1)}\partial y^{d_G^+(1)}}\cdots
\frac{\partial^{d_G(p)}f}{\partial x^{d_G^-(p)}\partial y^{d_G^+(p)}}
\end{align}
where $d_G^-(i), d_G^+(i)$ and  $d_G(i)$ are the outdegree,
the indegree and the degree of the vertex $i\in V(G)$,
respectively.
\item[(2)]
Let $G$ and $G'$ be the tographs associated to elements
$(i_1, \cdots, i_n; j_1, \cdots, j_n)$ and
$(i'_1, \cdots, i'_n$; $j'_1, \cdots, j'_n)\in \Gamma_n$,
respectively. Then
$$\Phi^{i_1,\cdots, i_n}_{j_1,\cdots, j_n}(f^{\ot p})
=\Phi^{i'_1,\cdots, i'_n}_{j'_1,\cdots, j'_n}(f^{\ot p})$$
for all $f\in A$ if and only if
there exists  a permutation $\sigma\in \mathbb{S}_n$ such that
$$(d_{G'}^+(i), d_{G'}^-(i))=(d_G^+(\sigma(i)), d_G^-(\sigma(i)))$$
for all $i=1, \cdots, p$.
\end{enumerate}
\end{lemma}

\begin{proof}
(1) By the definition of $\Phi^{i_1,\cdots, i_n}_{j_1,\cdots, j_n}$,
we immediately get the desired equality
\eqref{E10.4.1}.

(2) By (1), it is clear.
\end{proof}

For convenience, we denote
\[G(^{i_1,\cdots, i_n}_{j_1,\cdots, j_n})(f)
=\frac{\partial^{d_G(1)}f}{\partial x^{d_G^-(1)}\partial y^{d_G^+(1)}}\cdots
\frac{\partial^{d_G(p)}f}{\partial x^{d_G^-(p)}\partial y^{d_G^+(p)}},\]
and hence $\Phi^{i_1,\cdots, i_n}_{j_1,\cdots, j_n}(f^{\ot p})
=G(^{i_1,\cdots, i_n}_{j_1,\cdots, j_n})(f)$.

\begin{corollary}
\label{xxcor10.5}
Let $G$ and $G'$ be the tographs in $\mathcal{G}_n$.
If $G$ is equivalent to $G'$, then $G(f)=G'(f)$ for any $f\in \K[x, y]$.
\end{corollary}

\begin{remark}
The converse of Corollary \ref{xxcor10.5} does not hold and
a counter-example is $G=G({112\atop 234})$ and $G'=G({113\atop 234})$
when $p=5$ and $n=3$.
\end{remark}

\begin{proof}[Sketch Proof of Claim \ref{xxcl10.2}]
For each $(i_1, \cdots, i_n; j_1, \cdots, j_n)\in \Gamma_n$,
we denote by $N(^{i_1, \cdots, i_n}_{j_1, \cdots, j_n})$
the number of the tographs which are equivalent to
$G({{i_1, \cdots, i_n}\atop{j_1, \cdots, j_n}})$.

We consider the decomposition
\[G(^{i_1, \cdots, i_n}_{j_1, \cdots, j_n})
=G_{11}\cup \cdots \cup G_{1k_1}\cup \cdots \cup
G_{r1}\cup \cdots \cup G_{rk_r},\]
where $G_{is}$, for $1\le s\le k_i$ and $1\le i\le r$, are
connected components of $G$ with $G_{is}\cong G_{it}$ for all
$1\leq s,t \leq k_i$, and
$G_{is}\not\cong G_{jt}$ for $i\neq j$. Denote
$|V(G_{is})|=n_i$ for each $i=1, \cdots, r$.
{By definition, a tograph $G'$ is equivalent to $G(^{i_1, \cdots, i_n}_{j_1, \cdots, j_n})$,
if and only if for each $i=1, \cdots, r$, $G'$ admits $k_i$ connected components being isomorphic to $G_{i1}$.
Therefore, by combinatorial counting, we have that
\[N(^{i_1, \cdots, i_n}_{j_1, \cdots, j_n})
=\dfrac{p!}{(n_1!)^{k_1}\cdots (n_r!)^{k_r}k_1!\cdots k_r!}\]
for each $(i_1, \cdots, i_n; j_1, \cdots, j_n)\in \Gamma_n$.}
Clearly, if $n\le p-2$, then the underlying graph of
$G({{i_1, \cdots, i_n}\atop{j_1, \cdots, j_n}})$ is
not connected since
$|E(G({{i_1, \cdots, i_n}\atop{j_1, \cdots, j_n}}))|
=n<p-1=|V(G({{i_1, \cdots, i_n}\atop{j_1, \cdots, j_n}}))|-1$.
Therefore, $r\ge 2$ and $1\le k_t, n_t<p$ for each $t$.
Therefore, by \eqref{E10.1.3},
\begin{align*}
M_n^p(f)=& \frac{1}{n!}
\sum_{(i_1, \cdots, i_n; j_1, \cdots, j_n)\in \Gamma_n}
\Phi^{i_1, \cdots, i_n}_{j_1, \cdots, j_n}(f^{\ot p})\\
=& \frac{1}{n!}
\sum_{G\in \mathcal{G}_n} \dfrac{n!}{\prod_{1\le u<v\le p}\nu(u, v)!} G(f)\\
=& \frac{1}{n!}
\sum_{[G]\in \mathcal{G}_n/\sim }\dfrac{n!}{\prod_{1\le u<v\le p}\nu(u, v)!}
\dfrac{p!}{(n_1!)^{k_1}\cdots (n_r!)^{k_r}k_1!\cdots k_r!}G(f)\\
\equiv & 0 \ \ \ ({\rm mod}\ p)
\end{align*}
where the sum $\sum\limits_{[G]\in \mathcal{G}_n/\sim }$ means
that one take one element in each equivalence class of
$\mathcal{G}_n$ with respect to the relation $\sim$.
\end{proof}

\subsection*{Acknowledgments}
Both Y.-H. Bao and Y. Ye were supported by NSFC
(Grant No. 11401001, 11431010 and 11571329)
and J.J. Zhang by the US National Science
Foundation (grant No. DMS 1402863).

\vspace{0.5cm}

\end{document}